\documentclass[final]{llncs}

\usepackage{appendix} 
\usepackage{amsmath} 
\usepackage{multirow} 
\usepackage{tikz} 
\usepackage{changebar} 
\usepackage{ifdraft}
\usepackage{mathabx}
\usepackage{comment} 
\usepackage{txfonts}
\usepackage{enumitem} 
\usepackage{cleveref} 
\usepackage{cite} 
\usepackage{color}

\usepackage{tikz}
\usetikzlibrary{patterns}

\crefname{property}{property}{properties}
\crefname{axiom}{}{}

\newcommand\reducesto\to

\newcommand\N{\mathbb{N}}

\newcommand\Terms\Lambda

\newcommand\subst[2]{[ #1 \coloneqq #2 ]}
\newcommand\quant[2]{#1 #2.\ }

\newcommand\type\tau

\newcommand\num[1]{\boldsymbol{#1}} 
\newcommand\limply\to 
\newcommand\tarrow\to 

\newcommand\re{\mathrel{\mathcal{R}}} 
\newcommand\monRe{\mathrel{\mon{R}}} 
\newcommand\seq\vdash 
\newcommand\monSeq\Vdash 
\newcommand\mon\mathfrak 

\def\monUnit{}
\def\monStar{}
\def\monMerge{}
\newcommand\ifNotEmpty[2]{\ifthenelse{\equal{#1}{}}{}{#2}}
\newcommand\optionalSubscript[1]{\ifthenelse{\equal{#1}{}}{}{_{#1}}}
\newcommand\optionalSuperscript[1]{\ifthenelse{\equal{#1}{}}{}{^{#1}}}
\newcommand\setmonad[1]{%
\def\mm##1{\lVert ##1 \rVert\ifthenelse{\equal{#1}{}}{}{_{#1}}}%
\def\m##1{\lvert ##1 \rvert\ifthenelse{\equal{#1}{}}{}{_{#1}}}%
\def\monTrans##1{\llbracket ##1 \rrbracket\ifthenelse{\equal{#1}{}}{}{_{#1}}}%
\def\monRe{\mathrel{\mon R}\ifthenelse{\equal{#1}{}}{}{_{#1}}}%
\def\re{\mathrel{\mathcal R}\ifthenelse{\equal{#1}{}}{}{_{#1}}}%
\def\T{T\optionalSubscript{#1}}%
\renewcommand\monUnit[1][]{\operatorname{\mon{unit}}\optionalSubscript{#1}\optionalSuperscript{##1}}%
\renewcommand\monStar[1][]{\operatorname{\mon{star}}\optionalSubscript{#1}\optionalSuperscript{##1}}%
\renewcommand\monMerge[1][]{\operatorname{\mon{merge}}\optionalSubscript{#1}\optionalSuperscript{##1}}%
\def\monSeq{\Vdash\optionalSubscript{#1}}
}
\setmonad{}

\newcommand\EM{\mathrm{EM}_1} 
\newcommand\EMG{\mathrm{EM}} 

\DeclareMathOperator\abstrOp\Lambda
\DeclareMathOperator\unit\ast

\newcommand\newtermconstant[3]{\newcommand{#1}[1][]{#2
\ifthenelse{\equal{##1}{}}{}{^{##1}}%
\ifthenelse{\equal{#3}{}}{}{_\textrm{#3}}}%
}
\newcommand\newtermconstantN[2]{\newcommand{#1}[2][]{#2
\ifthenelse{\equal{##1}{}}{}{^{##1}}%
\ifthenelse{\equal{##2}{}}{}{_{##2}}}%
}

\newcommand\termname[1]{\operatorname{\textsf{\small#1}}}
\newtermconstant\pair{\termname{pair}}{}
\newtermconstant\prr{\termname{pr}}{R}
\newtermconstant\prl{\termname{pr}}{L}
\newtermconstant\inr{\termname{in}}{R}
\newtermconstant\inl{\termname{in}}{L}
\newtermconstantN\totRec{\operatorname{rec}}
\newtermconstantN\monStarN{\operatorname{\mon{star}}}
\newtermconstantN\monRaiseN{\operatorname{\mon{raise}}}
\newtermconstant\monBind{\operatorname{\mon{bind}}}{}

\newtermconstant\reg{\termname{reg}}{}
\newtermconstant\exc{\termname{ex}}{}
\newtermconstant\ite{\termname{ite}}{}
\newtermconstant\true{\termname{true}}{}
\newtermconstant\false{\termname{false}}{}
\newtermconstant\dummy{\termname{dummy}}{}
\newtermconstant\zero{\termname{zero}}{}
\let\succ\undefined
\newtermconstant\succ{\termname{succ}}{}

\usepackage{xcolor,ifthen}
\definecolorseries{parcolors}{hsb}{step}{red}{0.2,0,0}
\resetcolorseries{parcolors}
\newcounter{bracketcounter}
\newcommand\bracketcolor[1]{\textcolor{parcolors!![\arabic{bracketcounter}]}{#1}}
\newcommand\openpar{\addtocounter{bracketcounter}{1}\bracketcolor{(}}
\newcommand\closepar{\bracketcolor{)}\addtocounter{bracketcounter}{-1}}
\newcommand\parenthesis[1]{\linebreak[1] \openpar \ifthenelse{\arabic{bracketcounter} < 4}{#1}\ldots\closepar }

\newcommand\showTypeInfo{\renewcommand\typeInfoZ[2]{\ifthenelse{\isodd{\arabic{typeinfocounter}}}{\overbrace{##1}^{##2}}{\underbrace{##1}_{##2}}}}

\newcounter{typeinfocounter} 
\newcommand\typeInfoZ{}\showTypeInfo

\newcommand\fixme[1]{\ifoptiondraft{\textcolor{red}{ (X)}\marginpar{\textcolor{red}{#1}}}{}}

\usepackage{bussproofs}
\usepackage{tikz}
\usetikzlibrary{patterns}
\usetikzlibrary{decorations.pathmorphing}
 
\tikzstyle{label} = [fill=white,inner sep=1pt]
\newcommand\IfNotEmpty[2]{\ifthenelse{\equal{#1}{}}{}{#2}}

\newcommand\PrDer[3]{
\begin{tikzpicture}[baseline=0.3cm,scale=1]
  \IfNotEmpty{#3}{ 
    \draw decorate [decoration={snake,amplitude=2pt,segment length=14pt}] {(0,0.2) --  node[above,label] {\(#3\)} (0,1.25)};
  }
  \draw (-0.5,1) -- (-0.1,0) -- (0.1,0) -- (0.5,1);
  \IfNotEmpty{#1}{ 
    \node[above,label] () at (0,0.2) {\(#1\)}; 
  } 
    
\end{tikzpicture}
}
\newcommand\PrInfBr[2]{\def\extraVskip{0pt}\noLine\UnaryInfC{\PrDer{#1}{}{#2}}\noLine\def\extraVskip{2pt}}

\newcommand\PrAss[3][]{\AxiomC{\([#2\optionalSuperscript{#1}]^{#3}\)}}
\newcommand\PrAx[2][]{\AxiomC{\(#2\optionalSuperscript{#1}\)}}
\newcommand\PrUn[2][]{\UnaryInfC{\(#2\optionalSuperscript{#1}\)}}
\newcommand\PrBin[2][]{\BinaryInfC{\(#2\optionalSuperscript{#1}\)}}
\newcommand\PrTri[2][]{\TrinaryInfC{\(#2\optionalSuperscript{#1}\)}}
\newcommand\PrLbl[2][]{\LeftLabel{\(#2\)}\ifthenelse{\equal{#1}{}}{}{\RightLabel{\(#1\)}}}

\newcommand\PrInf[1][]{\ifthenelse{\equal{#1}{}}{%
\def\extraVskip{-2pt}\noLine\UnaryInfC\vdots\noLine\def\extraVskip{2pt}}{%
\noLine\UnaryInfC{\(#1\)}}}
\newcommand\RuleName[3][]{#2\mathrm{#3}_\mathrm{#1}}

\newenvironment{rulelisting}[1][\quad]{
\renewcommand\newline{\\[8mm] \hline}
\newcommand\nextrule[1][]{& \ifthenelse{\equal{##1}{}}{#1}{##1} &}
\newcommand\wholeline[1]{\multicolumn{3}{|c|}{##1}}
\newcommand\header[1]{\wholeline{##1} \\ }

\begin{displaymath}
\begin{array}{|ccc|}
\hline
}{
\newline
\end{array}
\end{displaymath}
}

\excludecomment{thesis}
\newcommand\onlyInThesis[1]{}
\includecomment{article}
\newcommand\onlyInArticle[1]{#1}

\newcommand\lfa{f}
\newcommand\lfb{g}
\newcommand\lfc{h}

\newcommand\lva{x}
\newcommand\lvb{y}
\newcommand\lvc{z}
\newcommand\lta{t} 
\newcommand\ltb{u} 
\newcommand\ltc{v} 
\newcommand\lafa{P}

\newcommand\fa{A}
\newcommand\fb{B}
\newcommand\fc{C}

\newcommand\lfalse\bot

\newcommand\riA\alpha
\newcommand\riB\beta

\title{A Witness Extraction Technique by Proof Normalization \\ Based on Interactive Realizability}
\author{Giovanni Birolo}
\institute{Department of Mathematics, University of Turin\\\email{giovanni.birolo@gmail.com}}

\begin{document}

\maketitle

\renewcommand\EM{\text{EM}_1}
\newcommand\foA{\mathfrak{a}}
\newcommand\foB{\mathfrak{b}}
\newcommand\foC{\mathfrak{c}}
\newcommand\brA\zeta
\newcommand\brB\eta
\newcommand\brC\eta
\newcommand\fd{D}
\newcommand\riC\gamma
\newcommand\brApp[2]{#1,#2}
\newcommand\derA\Pi
\newcommand\derB\Sigma

\newcommand\foa{\mathfrak{a}}
\newcommand\fob{\mathfrak{b}}
\newcommand\foc{\mathfrak{c}}
\newcommand\ri[1]{\ifthenelse{\equal{#1}\foA}\riA{%
    \ifthenelse{\equal{#1}\foB}\riB{%
      \ifthenelse{\equal{#1}\foC}\riC{}%
    }%
  }%
}
\newcommand\f[1]{\ifthenelse{\equal{#1}\foA}\fa{%
    \ifthenelse{\equal{#1}\foB}\fb{%
      \ifthenelse{\equal{#1}\foC}\fc{|{#1}|}%
    }%
  }%
}
\newcommand\der[1]{\check{#1}}

\newcommand\RuleNameEM[1][1]{\mathrm{EM}\optionalSubscript{#1}}
\newcommand\RuleNameIndE{\mathrm{Ind}}
\newcommand\RuleNameAtomI{\RuleName{\mathcal A}{I}}
\newcommand\RuleNameAtomE{\RuleName{\mathcal A}{E}}

\newcommand\RedNameProp[1]{\RuleName{{#1}}{-red}}
\newcommand\RedNamePerm[2][]{\RuleName{\ifthenelse{\equal{#1}{}}{#2}{{#2}{/}{#1}}}{-perm}}
\newcommand\RedNameSimpl[1]{\RuleName{{#1}}{-simpl}}

\newcommand\RedNameInd{\RedNameProp{\RuleNameIndE}}
\newcommand\RedNameWitness{\RedNameProp{\text{Wit}}}

\begin{article}
  \begin{abstract}
    We present a new set of reductions for derivations in natural deduction that can extract witnesses from closed derivations of simply existential formulas in Heyting Arithmetic (HA) plus the Excluded Middle Law restricted to simply existential formulas (\(\EM\)), a system motivated by its interest in proof mining. 

    The reduction we have for classical logic are quite different from all existing ones. They are inspired by the informal idea of learning by making falsifiable hypothesis and checking them, and by the Interactive Realizability interpretation. We extract the witnesses directly from derivations in HA+\(\EM\) by reduction, without encoding derivations by a realizability interpretation. 
  \end{abstract}
\end{article}
\section{Introduction}

It has been known since the seminal works of G\"odel and Kreisel that classical proofs of simply existential statements have a computational content. 
Since 1990, we know that they may be interpreted as functional programs with continuations \cite{griffin90}.
A more recent technique for extracting this content in the case of the sub-classical logic HA+\(\EM\) (Heyting Arithmetic + Excluded Middle Law restricted to \(\Sigma^0_1\) formulas) is Interactive Realizability \cite{aschieriB10}, introduced by Berardi and de'Liguoro in \cite{berardidL08,berardidL09}. 

\subsection{Interactive Realizability}
Interactive realizability combines the ideas of realizability semantics for classical logic with the ideas of monotonic learning. 
It extends Kreisel modified realizability for intuitionistic logic by realizing \(\EM\). 
This is accomplished with realizers that learn witnesses for \(\Sigma^0_1\) formulas by trial and error. 
When evaluated, an interactive realizer may fail to produce correct evidence for the statement it realizes. 
In particular it may fail to provide a witness for an existential statement or to correctly decide a disjunction. 
In this case, some error is found and the realizer provides a new piece of knowledge which is added to a knowledge state. 
This knowledge state will be used when the interactive realizer is evaluated again:
if the state contains all the necessary knowledge then the realizer will succeed and provide correct evidence. 
Otherwise it will fail again and a new piece of knowledge will be added to the state. 
Berardi and Aschieri in \cite{aschieriB10} prove by a continuity argument that this process converges, 
that is, an interactive realizer may only fail a finite number of times before succeeding. 

Traditional continuation-based techniques often extract algorithms that are obscure and may fail to convey the intuitive ideas of the proof. 
Interactive realizability aims to explain the use of continuations as a particular implementation of the more general concept of learning. 
In order to reach this goal interactive realizers have the following features:
\begin{itemize}
  \item they realize statements directly without the need for a negative translation;
  \item they encode the proof faithfully, for instance avoiding blind searches for witnesses;
  \item 
    when the excluded middle instances in the proof have a low logical complexity, like \(\EM\), 
    interactive realizers use simpler constructs like states and exceptions to handle the continuation-passing nature of the computational content of classical proofs. 
\end{itemize}
The goal of this paper is to express thorough reductions the idea of trial and error from Interactive Realizability. 

\subsection{Witness extraction}
In proof theory there are reductions that express the computational interpretation we give to logical connectives, quantifiers and, in the case of arithmetic, induction.
Proofs in intuitionistic logic are shown to produce a witness for existential statements:
any proof can be to normal form, in which no more reductions are possible,  
and in a normal proof of an existential statement a witness always appears in a predictable location.
We want to obtain the same result for proofs of semi-decidable statements in intuitionistic logic augmented with \(\EM\) and reduction rules inspired by a trial-and-error interpretation. 

We work in Heyting Arithmetic (HA) extended with \(\EM\), which is weaker than classical arithmetic but strong enough to prove non-trivial non-constructive results: 
for instance the fact that every function \(f : \N \to \N\) has a minimum. 
By modifying the standard reductions for Heyting Arithmetic (see \cite{prawitz71}), we show that normal proofs of existential statements in \(\text{HA}+\EM\) produce a witness\footnote{Under suitable assumptions on the proof.}, as they do in the intuitionistic case. 

The fact that classical arithmetic is a conservative extension of HA for \(\Pi^0_2\) statements is well known and the fact that we can extract witnesses from classical proofs of \(\Sigma^0_1\) statements follows immediately. 
However proofs of these results usually employ the G\"odel-Gentzen negative translation combined with variants of Kreisel's modified realizability semantics or Friedman's translation. 
Here, by purely proof theoretical means, we prove a slightly weaker result without resorting to negative translations and 
using reductions justified in terms of Interactive Realizability. 

\subsection{Future work}
We shall prove in a future paper that our reductions are strongly normalizing. 
In this paper we prove that, \emph{if} we have normalization, then all derivations of simply existential statements compute a witness by a method we describe as trial-and-error. 

\section{A formal system for Intuitionistic Arithmetic}
In this section we recount some basic features of Intuitionistic Arithmetic HA: for a more throughout account see \cite{troelstra73}.

The language of HA is a first-order language with connectives \(\land, \lor, \limply\) and quantifiers \(\forall\) and \(\exists\).
\begin{article}
  As usual we say that a formula is \emph{closed} if it has no free variables and that a rule is \emph{atomic} if its premisses and its conclusion can only be atomic formulas\footnote{An instance of the induction rule can have atomic premisses and conclusion, but this is not required in general, so the induction rule is not atomic.}.
\end{article}

The language of HA includes: 
\begin{itemize}
  \item variables \(\lva, \lvb, \lvc, \dotsc\) for natural numbers,
  \item function symbols for all the primitive recursive functions,
  \item the equality binary predicate symbol \(=\), 
  \item arithmetical terms (metavariables \(\lta, \ltb, \ltc\)), atomic formulas (\(\lafa\)) and formulas (\(\fa, \fb, \fc\)), defined inductively as usual.
\end{itemize}
In particular we assume that we have the function symbols \(\num 0, \succ \).

A \emph{numeral} is any term of the form \(\succ^n(\num 0)\), for some \(n \in \N\). 
We assume having a set of algebraic reduction rules for primitive recursive functions. 
If, for some \(\lfb : \N^n \to \N\) and \(\lfc : \N^{n+2} \to \N\), \(\lfa\) denotes the primitive recursive \((n{+}1)\)-ary function defined by the equations:
\begin{align*}
  \lfa(0, \lva_1, \dotsc, \lva_n) &= \lfb(\lva_1, \dotsc, \lva_n), \\ 
  \lfa(\succ(\lva), \lva_1, \dotsc, \lva_n) &= \lfc(\lva, \lfa(\lva, \lva_1, \dotsc, \lva_n), \lva_1, \dotsc, \lva_n),
\end{align*}
then we add the reductions:
\begin{align*}
  \lfa(0, \lva_1, \dotsc, \lva_n) &\reducesto \lfb(\lva_1, \dotsc, \lva_n), \\ 
  \lfa(\succ(\lva), \lva_1, \dotsc, \lva_n) &\reducesto \lfc(\lva, \lfa(\lva, \lva_1, \dotsc, \lva_n), \lva_1, \dotsc, \lva_n).
\end{align*}
This reduction system is strongly normalizing and has an unique normal form. 
Thus any closed normal term is a numeral. 
We reduce terms inside formulas and we consider two formulas equal when they have the same normal form. 

We assume that we have a recursive set of atomic rules that is correct (that is, it derives true statements only) and contains:
\begin{itemize}
  \item the first-order axioms and rules for equality, 
  \item the compatibility rule for equality and functions,
  \item the \emph{Ex Falso Quodlibet} rule for atomic formulas \(\lafa\):
    \[ 
      \PrAx\lfalse 
      \PrLbl{\RuleName\lfalse{E}} 
      \PrUn\lafa 
      \DisplayProof 
    \] 
    where \(\lfalse\) denotes the \(0\)-ary relation that never holds; 
  \item the rules for zero and the successor functions:
    \[
      \PrAx{\succ(\lva) = \num 0}
      \PrUn\lfalse
      \DisplayProof
      \qquad
      \PrAx{\succ(\lva) = \succ(\lvb)}
      \PrUn{\lva = \lvb}
      \DisplayProof
    \]
\end{itemize}
The non-atomic inference rules are just those of minimal first-order logic, that is, one elimination and one introduction rule for each of the logical connectives and quantifiers and the induction rule, which we define later. 
Moreover we consider \(\lnot \fa\) syntactic sugar for \(\fa \limply \lfalse\).

Note that minimal logic extended with the Ex Falso Quodlibet rule for atomic formulas yields full intuitionistic logic. 

Since our reduction technique could conceivably be used in other first-order theories, 
we isolate some general assumptions on atomic formulas and rules that we need for our results to hold:
\begin{itemize}
  \item closed atomic formulas are decidable,
  \item any true closed atomic formula has an atomic derivation,
  \item atomic rules do not discharge assumptions,
  \item atomic rules do not bind term variables\footnote{The precise meaning of this will be made precise later.}. 
\end{itemize}
The first two assumption are very reasonable in a constructive setting such as arithmetic where we expect to have decidability at least for atomic formulas\footnote{However they may very well fail in set theory, for instance with the inclusion predicate.}. 
The other two seems also reasonable for any first-order theory. 
All of them are satisfied in HA. 
These assumption are reasonable in a constructive setting and they are satisfied in Heyting Arithmetic

We assumed that any true closed atomic formula has an atomic derivation. 
For convenience we add atomic rules for proving them in one step. 
Let \(\lafa\) be a closed atomic formula. 
If \(\lafa\) is true then we add the atomic axiom:
\[
  \PrAx{}
  \PrLbl{\RuleNameAtomI}
  \PrUn\lafa
  \DisplayProof
\]
Otherwise if \(\lafa\) is false we add the atomic rule:
\[
  \PrAx\lafa
  \PrLbl{\RuleNameAtomE}
  \PrUn\bot
  \DisplayProof
\]

In order to work on the structure of derivations we need suitable notation and terminology. 
We represent derivations as upward growing trees of formulas and we make a distinction between a formula (risp. rule) and its occurrences (risp. instances) in a derivation (see \Cref{app:occurrences_instances} for a more precise description). 
We define \emph{assumptions} and \emph{open assumptions} as usual in natural deduction, see \cite{troelstra00}, page 23.

\section{The standard reductions}

\begin{thesis}
  Reductive proof theory stems from the following observation: 
  there are derivations that are more complex than they need to be because they have unnecessary detours.
  Moreover we can produce simpler and more direct derivations with the same conclusion by simple structural manipulations called \emph{reductions}.
\end{thesis}

In standard reductive proof theory for natural deduction, several reductions are introduced: proper reductions, permutative reductions, immediate simplifications and a reduction for the induction rule (see \cite{prawitz71}).
A derivation is said to be \emph{fully normal} when none of these reductions can be performed on it. 
For our purposes fully normal derivations are not required, so we introduce only the proper reductions and the induction reduction. 

In an instance of an elimination rule, the premiss containing the connective or quantifier that is being eliminated is called the \emph{major} premiss; the other premisses are called the \emph{minor} premisses. 
We always display the major premiss in the leftmost position. 

\subsection{Proper reductions} 
Consider a derivation in which a formula occurrence \(\foA\) is both the conclusion of an introduction rule instance \(\riA\) and the major premiss of an elimination rule instance \(\riB\). 
Then we can derive the conclusion of \(\riB\) directly by removing \(\riA\) and \(\riB\) and rearranging the derivations of the premisses of \(\riA\) and of the minor premisses of \(\riB\) (if any).
Note that \(\riA\) and \(\riB\) must be instances of an introduction rule and an elimination rule for the same logical connective, since the formula introduced by \(\riA\) is the same formula eliminated by \(\riB\). 
Therefore for each logical connective we have a different type of \emph{proper reduction}. 
They are listed in \Cref{fig:proper_reductions}. 

\begin{figure}[h!]
  \caption{The proper reductions.}
  \label{fig:proper_reductions}
  \begin{rulelisting}[\reducesto]
    \header{\RedNameProp\land} 
    \PrAx{}\PrInf[\derA_1]
    \PrUn{\fa_1}
    \PrAx{}\PrInf[\derA_2]
    \PrUn{\fa_2}
    \PrLbl[\riA]{\RuleName\land{I}}
    \PrBin{\fa_1 \land \fa_2}
    \PrLbl[\riB]{\RuleName\land{E}}
    \PrUn{\fa_i}
    \DisplayProof 
    \nextrule[\xrightarrow{\RedNameProp\land}]
    \PrAx{}\PrInf[\derA_i]
    \PrUn{\fa_i}
    \DisplayProof
    \newline
    \header{\RedNameProp\lor} 
    \PrAx{}\PrInf[\derA]
    \PrUn{\fa_i}
    \PrLbl{\RuleName\lor{I}}
    \PrUn{\fa_1 \lor \fa_2}
    \PrAss{\fa_1}\riA
    \PrInf[\derA_1]
    \PrUn\fb
    \PrAss{\fa_2}\riA
    \PrInf[\derA_2]
    \PrUn\fb
    \PrLbl[\riA]{\RuleName\lor{E}}
    \PrTri\fb
    \DisplayProof 
    \nextrule[\xrightarrow{\RedNameProp\lor}]
    \PrAx{}\PrInf[\derA]
    \PrUn{\fa_i}
    \PrInf[\derA_i]
    \PrUn\fb
    \DisplayProof
    \newline
    \header{\RedNameProp\limply} 
    \PrAss\fa\riA
    \PrInf[\derA_1]
    \PrUn\fb
    \PrLbl[\riA]{\RuleName\limply{I}}
    \PrUn{\fa \limply \fb}
    \PrAx{}
    \PrInf[\derA_2]
    \PrUn\fa
    \PrLbl{\RuleName\limply{E}}
    \PrBin\fb
    \DisplayProof 
    \nextrule[\xrightarrow{\RedNameProp\limply}]
    \PrAx{}\PrInf[\derA_2]
    \PrUn\fa
    \PrInf[\derA_1]
    \PrUn\fb
    \DisplayProof
    \newline
    \header{\RedNameProp\forall} 
    \PrAx{}
    \PrInf[\derA]
    \PrLbl{\RuleName\forall{I}}
    \PrUn{\quant\forall\lva\fa}
    \PrLbl{\RuleName\forall{E}}
    \PrUn{\fa\subst\lva\lta}
    \DisplayProof 
    \nextrule[\xrightarrow{\RedNameProp\forall}]
    \PrAx{}
    \PrInf[\derA\subst\lva\lta]
    \PrUn{\fa\subst\lva\lta}
    \DisplayProof
    \newline
    \header{\RedNameProp\exists}
    \PrAx{}
    \PrInf[\derA_1]
    \PrUn{\fa\subst\lva\lta}
    \PrLbl{\RuleName\exists{I}}
    \PrUn{\quant\exists\lva\fa}
    \PrAss{\fa\subst\lva\lvb}\riA
    \PrInf[\derA_2]
    \PrUn\fb
    \PrLbl[\riA]{\RuleName\exists{E}}
    \PrBin\fb
    \DisplayProof 
    \nextrule[\xrightarrow{\RedNameProp\exists}]
    \PrAx{}
    \PrInf[\derA_1]
    \PrUn{\fa\subst\lva\lta}
    \PrInf[\derA_2\subst\lvb\lta]
    \PrUn\fb
    \DisplayProof
  \end{rulelisting}
\end{figure}

\subsection{Induction reduction}
Consider the induction rule schema \(\RuleNameIndE\) in the following form:
\[ 
  \PrAx{}
  \PrInf[\derA_1]
  \PrUn{\fa\subst\lva{\num 0}}
  \PrAss\fa\riA
  \PrInf[\derA_2]
  \PrUn{\fa\subst\lva{\succ(\lva)}}
  \PrLbl[\riA]\RuleNameIndE
  \PrBin{\fa\subst\lva\lta}
  \DisplayProof
\] 
We call \(\lta\) the \emph{main term} of the induction. 
\onlyInThesis{\fixme{relate this to the induction definition given in prelims}} 
An instance \(\riA\) of the \(\RuleNameIndE\) rule can be reduced when the main term \(\lta\) in its conclusion \(\fa\subst\lva\lta\) is either \(\num 0\) or \(\succ(\ltb)\) for some term \(\ltb\). 
Then if \(\lta = \num 0\) we can reduce \(\riA\) to: 
\[ 
  \PrAx{} 
  \PrInf[\derA_1] 
  \PrUn{\fa\subst\lva{\num 0}} 
  \DisplayProof 
\]
and if \(\lta = \succ(\ltb)\) as: 
\[ 
  \PrAx{}
  \PrInf[\derA_1]
  \PrUn{\fa\subst\lva{\num 0}}
  \PrAss\fa\riB
  \PrInf[\derA_2]
  \PrUn{\fa\subst\lva{\succ(\lva)}}
  \PrLbl[\riB]{\RuleNameIndE}
  \PrBin{\fa\subst\lva\ltb}
  \PrInf[\derA_2\subst\lva\ltb]
  \PrUn{\fa\subst\lva{\succ(\ltb)}}
  \DisplayProof
\]
We call this conditional reduction \(\RedNameInd\). 
It is easy to see that this reduction is ``unraveling'' the induction. 
When \(\ltb\) is a numeral \(\num n\), that is, a term of the form \(\succ^n\), we can apply the \(\RedNameInd\) reduction repeatedly (\(n\) times) until we remove all occurrences of the \(\RuleNameIndE\) rule and get:
\[
  \PrAx{}
  \PrInf[\derA_1]
  \PrUn{\fa\subst\lva{\num 0}}
  \PrInf[\derA_2\subst\lva{\num 0}]
  \PrUn{\fa\subst\lva{\num 1}}
  \PrInf[\derA_2\subst\lva{\num 1}]
  \PrUn{\fa\subst\lva{\num 2}}
  \PrInf
  \PrUn{\fa\subst\lva{\num n}}
  \DisplayProof
\]



\section{The Witness Extracting Reductions}

In this section we introduce an inference rule that is equivalent to 
\onlyInArticle{a restricted version of the excluded middle axiom schema}
\onlyInThesis{the restricted excluded middle axiom schema \(\EM\) defined in \Cref{def:excluded_middle}}
and two reductions involving this new rule. 
The first one, the \(\RedNameWitness\) reduction, is inspired to Interactive Realizability and it will be instrumental in converting classical derivations into constructive ones.
The second one is a permutative reduction and is needed later for technical reasons. 

\subsection{The \(\EM\) rule}
\begin{article}
  The excluded middle axiom schema is the following:
  \[ 
    \PrAx{}
    \PrLbl\EMG
    \PrUn{\fa \lor \lnot \fa} 
    \DisplayProof
  \]
  where \(\fa\) is any formula. 
  In intuitionistic logic with \(\EMG\) we can derive the double negation elimination. 
  This shows that HA extended with \(\EMG\) is equivalent to classical (Peano) arithmetic PA.

  We restrict the \(\EMG\) axiom to \(\Pi^0_1\) formulas
  and we rewrite it in an equivalent way, as in \cite{akamaBHK04}:
  \[ 
    \PrAx{}
    \PrLbl\EM
    \PrUn{\quant\forall\lva \lafa \lor \quant\exists\lva \lnot\lafa }
    \DisplayProof
  \]
  where \(\lafa\) is an atomic formula. 
\end{article}

For convenience we replace the \(\EM\) axiom schema with the equivalent \(\EM\) rule:
\[
  \PrAss{\quant\forall\lva\lafa}\riA
  \PrInf
  \PrUn\fa
  \PrAss{\lnot\lafa\subst\lva\lvb}\riA
  \PrInf
  \PrUn\fa
  \PrLbl[\riA]\EM
  \PrBin\fa
  \DisplayProof
\]
where the variable \(\lvb\) does not occur in \(\fa\) nor in any open assumption that \(\fa\) depends on except occurrences of the assumption \(\lnot \lafa\subst\lva\lvb\) (as in the \(\RuleName\exists{E}\) rule). 

The \(\EM\) rule can be easily derived by the \(\EM\) axiom by means of the \(\RuleName\lor{E}\) and \(\RuleName\exists{E}\) rules and it can easily derive the \(\EM\) axiom. 

The derivations are given in \Cref{sec:em_axiom_rule_equiv}.

In the following we refer to the assumption \(\quant\forall\lva\lafa\) in the derivation of the leftmost premiss of the \(\EM\) rule as the \emph{universal assumption} and to the assumption \(\lnot\lafa\subst\lva\lvb\) in the derivation of the rightmost premiss as the \emph{existential assumption}.

\begin{thesis}
  We can also write the \(\EM\) rule in sequent style as: 
  \[
    \PrAx{\Gamma, \riA : \quant\forall\lva \lafa \seq \fa}
    \PrAx{\Gamma, \riA : \quant\exists\lva \lnot \lafa \seq \fa}
    \PrLbl[\riA]\EM
    \PrBin{\Gamma \seq \fa}
    \DisplayProof
  \]
\end{thesis}

The universal assumption \(\quant\forall\lva\lafa\) is a \(\Pi^0_1\) formula and thus \emph{negatively decidable}, meaning that a finite piece of evidence is enough to prove it false: 
a counterexample, a natural number \(m\) such that \(\lafa\subst\lva{\num m}\) does not hold. 
Moreover, if we know that it is false, then a counterexample exists and we can find it in a finite time, in the worst case by means of a blind search through all the natural numbers. 

On the other hand, in order to prove the universal assumption, we need a possibly infinite evidence, namely, we may need to check \(\lafa\subst\lva{\num m}\) for all natural numbers \(m\) and this cannot be effectively done (at least when we have no information on \(\lafa\)). 

The existential assumption \(\lnot\lafa\subst\lva\lvb\) is not actually a existential formula. 
However it is easy to see that it takes the place of the assumption discharged by the \(\RuleName\exists{E}\) rule. 
For more details you can see how the two rules are related in \Cref{sec:em_axiom_rule_equiv}.

We say that we can prove the existential assumption true by showing a witness, namely a number \(m\) such that \(\lnot\lafa\subst\lva{\num m}\). 
Thus the existential assumption behaves as if it were \emph{positively decidable}: 
when it is true, we have a terminating algorithm to find the finite evidence needed to prove it. 
However, when it false, we have no way to effectively decide if it is false. 

Note that a counterexample \(m\) for the universal assumption \(\quant\forall\lva\lafa\) is a witness for the existential assumption since in that case \(\lnot\lafa\subst\lva{\num m}\) holds.


\subsection{Witness reduction}

Consider a derivation \(\derA\) ending with an instance \(\riA\) of the \(\EM\) rule for the atomic formula \(\lafa\):
\[
  \PrAss{\quant\forall\lva\lafa}\riA
  \PrInf[\derA_1]
  \PrUn\fa
  \PrAss{\lnot\lafa\subst\lva\lvb}\riA
  \PrInf[\derA_2]
  \PrUn\fa
  \PrLbl[\riA]\EM
  \PrBin\fa
  \DisplayProof
\]

A priori we do not know any counterexample to the universal assumption (we do not even know whether it holds or not), so we begin by looking at how the assumption is used in \(\derA_1\). 
In \(\derA_1\), consider all the instances \(\riB_1, \dotsc, \riB_n\) of the \(\RuleName\forall{E}\) rule whose premiss is an occurrence of the universal assumption \(\quant\forall\lva\lafa\) and whose conclusion is the occurrence of a closed (atomic) formula:
\[ 
  \PrAss{\quant\forall\lva\lafa}\riA
  \PrLbl[\riB_1]{\RuleName\forall{E}}
  \PrUn{\lafa\subst\lva{\lta_1}}
  \PrInf
  \DisplayProof \quad \dotso \quad
  \PrAss{\quant\forall\lva\lafa}\riA
  \PrLbl[\riB_n]{\RuleName\forall{E}}
  \PrUn{\lafa\subst\lva{\lta_n}}
  \PrInf
  \DisplayProof 
\]
These represent the concrete instances of the universal assumption that are used to derive \(\fa\) in \(\derA_1\). 
Since the conclusions of \( \riB_1, \dotsc, \riB_n\) are closed atomic formulas they are decidable. 
Therefore we can derive the true concrete instances directly with the atomic axiom \(\RuleNameAtomI\) instead of deducing them from the universal assumption. 
We distinguish two cases.
\begin{itemize}
  \item
    If \(\lafa\subst\lva{\lta_i}\) is true for all \(i\) we replace each \(\riB_i\) with the atomic axiom for \(\lafa\subst\lva{\lta_i}\):
    \[ 
      \PrAss{\quant\forall\lva\lafa}\riA
      \PrLbl[\riB_i]{\RuleName\forall{E}}
      \PrUn{\lafa\subst\lva{\lta_i}}
      \PrInf
      \DisplayProof \quad \leadsto \quad 
      \PrAx{}
      \PrLbl{\RuleNameAtomI}
      \PrUn{\lafa\subst\lva{\lta_i}}
      \PrInf
      \DisplayProof
    \] 
    We call this new derivation \(\derA_1'\). 

    Now two situations are possible: either \(\derA_1\) needs the universal assumption only to deduce the concrete instances \(\riB_1, \dotsc, \riB_n\) or not. 

    \begin{itemize}
      \item The first case happens when \(\derA_1'\) contains no more occurrences of the universal assumption discharged by \(\riA\),
        that is, the universal assumption only occurs in \(\derA_1\) as the premiss of \(\riB_1, \dotsc, \riB_n\). 
        In this case \(\derA_1'\) is a self-contained derivation of \(\fa\) and we can replace the whole \(\derA\) with \(\derA_1'\).
      \item Otherwise \(\derA_1'\) still contains some occurrence  of the universal assumption.
        Then \(\derA_1'\) does need the universal assumption itself and not just some concrete instances of it. 
        In this case we can only replace \(\derA_1\) with \(\derA_1'\) in \(\derA\), but we cannot eliminate the \(\EM\) rule instance \(\riA\) from the derivation. 
    \end{itemize}

  \item 
    Otherwise there is some \(i\) such that \(\lafa\subst\lva{\lta_i}\) is false. 
    This means that the universal assumption itself is false, since we have found the counterexample \(\lta_i\). 
    Moreover \(\lta_i\) is a witness for the existential assumption, meaning that 
    we can replace \(\lvb\) with \(\lta_i\) in \(\derA_2\) and all the occurrences of the assumption \(\lnot\lafa\subst\lva\lvb\) with a derivation of \(\lnot\lafa\subst\lva\lta\):  
    \[ 
      \PrAss{\lnot\lafa\subst\lva{\lta_i}}\riA
      \PrInf 
      \DisplayProof \quad \leadsto \quad 
      \PrAss{\lafa\subst\lva{\lta_i}}{\riB_i'} 
      \PrLbl{\RuleNameAtomE} 
      \PrUn\bot 
      \PrLbl[\riB_i']{\RuleName\limply{I}} 
      \PrUn{\lnot\lafa\subst\lva{\lta_i}} 
      \PrInf 
      \DisplayProof 
    \] 
    We call this new derivation \(\derA_2'\). 

    Note that in this case we replace all the occurrences of the existential assumption in \(\derA_2\) and thus \(\derA_2'\) is self-contained derivation of \(\fa\). 
    Therefore we can replace \(\derA\) with \(\derA_2'\).
\end{itemize}
We call this reduction \(\RedNameWitness\).

The gist of the \(\RedNameWitness\) reduction is that we look for counterexamples to the universal assumption in \(\derA_1\). 
If we do not find one then we have checked that all the concrete instances of the universal assumption hold.
Moreover if \(\derA_1\) uses the universal assumption exclusively to deduce these concrete instances, then we get a direct derivation of \(\fa\) without using the \(\EM\) rule.
On the other hand if we find a counterexample then we know that we can put it in \(\derA_2\) and get another direct derivation of \(\fa\).

In some sense we have a procedure to decide which one of the subderivation of the \(\EM\) rule is the effective one, 
Note that this procedure fails when we do not find counterexamples to the universal assumption but we cannot completely eliminate its occurrences from \(\derA_1\). 
Our main result can be thought of as the proof that, when the conclusion of a derivation is simply existential, this 
``failure'' of the procedure
does not happen. 
The whole reduction is summarized in \Cref{fig:em_reduction}.

\begin{figure}[h!]
  \caption{The \(\RedNameWitness\) reduction possible outcomes.}
  \label{fig:em_reduction}
  \begin{rulelisting}
    \header{\RedNameWitness}
    \hline
    \wholeline{
      \begin{tikzpicture}
        \node (forall 1) at (0,2) [] {
          \PrAss{\quant\forall\lva\lafa}\riA
          \PrLbl[\riB_1]{\RuleName\forall{E}}
          \PrUn{\lafa\subst\lva{\lta_1}}
          \DisplayProof
        };
        \node (forall n) at (3,2) [] {
          \PrAss{\quant\forall\lva\lafa}\riA
          \PrLbl[\riB_n]{\RuleName\forall{E}}
          \PrUn{\lafa\subst\lva{\lta_n}}
          \DisplayProof
        };
        \node (forall x) at (5.5,2) [] {
          \PrAss{\quant\forall\lva\lafa}\riA
          \DisplayProof
        };
        \node (der1) at (0.93,0.25) [inner sep=0pt] {}; 

        \node (root) at (4,0) [inner sep=0pt] {
          \PrAx{}
          \PrInf[\derA_1]
          \PrUn\fa
          \PrAx{\hspace{4cm}}
          \PrAss{\lafa\subst\lva\lvb}\riA
          \PrInf[\derA_2]
          \PrUn\fa
          \PrLbl[\riA]\EM
          \PrTri\fa
          \DisplayProof
        };
        \draw [dotted] (forall 1) -- (forall n);
        \draw [dotted] (forall n) -- (forall x);
        \draw [dotted] (forall 1) -- (der1);
        \draw [dotted] (forall n) -- (der1);
        \draw [dotted] (forall x) -- (der1);
      \end{tikzpicture}
    } 
    \\
    \header{\text{An derivation ending with an \(\EM\) rule instance reduces to:}}
    \hline
    \wholeline{
      \begin{tikzpicture}
        \node (forall 1) at (0,2) [] {
          \PrAx{}
          \PrLbl{\RuleNameAtomI}
          \PrUn{\lafa\subst\lva{\lta_1}}
          \DisplayProof
        };
        \node (forall n) at (3,2) [] {
          \PrAx{}
          \PrLbl\RuleNameAtomI
          \PrUn{\lafa\subst\lva{\lta_n}}
          \DisplayProof
        };
        \node (forall x) at (5.5,2) [] {
          \PrAss{\quant\forall\lva\lafa}\riA
          \DisplayProof
        };
        \node (der1) at (0.93,0.25) [inner sep=0pt] {}; 

        \node (root) at (4,0) [] {
          \PrAx{}
          \PrInf[\derA_1]
          \PrUn\fa
          \PrAx{\hspace{4cm}}
          \PrAss{\lafa\subst\lva\lvb}\riA
          \PrInf[\derA_2]
          \PrUn\fa
          \PrLbl[\riA]\EM
          \PrTri\fa
          \DisplayProof
        };
        \draw [dotted] (forall 1) -- (forall n);
        \draw [dotted] (forall n) -- (forall x);
        \draw [dotted] (forall 1) -- (der1);
        \draw [dotted] (forall n) -- (der1);
        \draw [dotted] (forall x) -- (der1);
      \end{tikzpicture}
    }
    \\
  \header{\begin{tabular}{c}when all \(\lafa\subst\lva{\lta_i}\) hold and some occurrences of the universal assumption remain.\end{tabular}}
    \hline
    \wholeline{
      \begin{tikzpicture}
        \node (forall 1) at (0,2) [] {
          \PrAx{}
          \PrLbl{\RuleNameAtomI}
          \PrUn{\lafa\subst\lva{\lta_1}}
          \DisplayProof
        };
        \node (forall n) at (3,2) [] {
          \PrAx{}
          \PrLbl\RuleNameAtomI
          \PrUn{\lafa\subst\lva{\lta_n}}
          \DisplayProof
        };
        \node (der1) at (0.93,0.25) [inner sep=0pt] {}; 

        \node (root) at (1,0) [] {
          \PrAx{}
          \PrInf[\derA_1]
          \PrUn\fa
          \DisplayProof
        };
        \draw [dotted] (forall 1) -- (forall n);
        \draw [dotted] (forall 1) -- (root);
        \draw [dotted] (forall n) -- (root);
      \end{tikzpicture}
    }
    \\
  \header{\begin{tabular}{c}when all \(\lafa\subst\lva{\lta_i}\) hold and no occurrence of the universal assumption remains.\end{tabular}}
    \hline
    \wholeline{
\PrAss{\lafa\subst\lva{\lta_i}}{\riB'} 
      \PrLbl{\RuleNameAtomE} 
      \PrUn\bot 
      \PrLbl[\riB']{\RuleName\limply{I}} 
      \PrUn{\lnot\lafa\subst\lva{\lta_i}} 
      \PrInf[\derA_2]
      \PrUn\fa
      \DisplayProof
    } \\ 
    \wholeline{\text{when some \(\lafa\subst\lva{\lta_i}\) does not hold.}}
  \end{rulelisting}
\end{figure}

%
%
%

\subsection{Permutative reduction for \(\EM\)}

The permutative reduction for \(\EM\) is defined in the same way as the permutative reduction for the \(\RuleName\lor{E}\) rule, 
that is, 
when the conclusion of a \(\EM\) rule instance is the major premiss of an elimination rule instance \(\RuleName\ast{E}\):
\[
  \PrAss{\quant\forall\lva\lafa}\riC
  \PrInf[\derA_1]
  \PrUn\fa
  \PrAss{\lnot\lafa\subst\lva\lvb}\riC
  \PrInf[\derA_2]
  \PrUn\fa
  \PrLbl[\riC]\EM
  \PrBin\fa
  \PrAx{}
  \PrInf[\bar\derA]
  \PrLbl{\RuleName\ast{E}}
  \PrBin\fb
  \DisplayProof
\]
reduces to:
\[
  \PrAss{\quant\forall\lva\lafa}\riA
  \PrInf[\derA_1]
  \PrUn\fa
  \PrAx{\bar\derA}
  \PrLbl{\RuleName\ast{E}}
  \PrBin\fb
  \PrAss{\lnot\lafa\subst\lva\lvb}\riA
  \PrInf[\derA_2]
  \PrUn\fa
  \PrAx{\bar\derA}
  \PrLbl{\RuleName\ast{E}}
  \PrBin\fb
  \PrLbl[\riA]\EM
  \PrBin\fb
  \DisplayProof
\]
where \(\bar\derA\) stands for the derivations of the remaining minor premisses of \(\riB\) if any. 
We denote this reduction as \(\RedNamePerm\EM\). 
More explicitly, we can define a permutative reduction for each elimination rule, see \Cref{fig:em_perms_connectives} and \Cref{fig:em_perms_quantifiers}. 
\begin{figure}[h]
  \caption{The permutative reductions of the \(\EM\) rule with the \(\RuleName\land{E}, \RuleName\lor{E}\) and \(\RuleName\limply{E}\) rules.}
  \label{fig:em_perms_connectives}
  \begin{rulelisting}[\reducesto]
    \header{\RedNamePerm[\land]\EM}
    \PrAss{\quant\forall\lva\lafa}\riA
    \PrInf[\derA_1]
    \PrUn{\fa_1 \land \fa_2}
    \PrAss{\lnot\lafa\subst\lva\lvb}\riA
    \PrInf[\derA_2]
    \PrUn{\fa_1 \land \fa_2}
    \PrLbl[\riA]\EM
    \PrBin{\fa_1 \land \fa_2}
    \PrLbl{\RuleName\land{E}}
    \PrUn{\fa_i}
    \DisplayProof
    \nextrule
    \PrAss{\quant\forall\lva\lafa}\riA
    \PrInf[\derA_1]
    \PrUn{\fa_1 \land \fa_2}
    \PrLbl{\RuleName\land{E}}
    \PrUn{\fa_i}
    \PrAss{\lnot\lafa\subst\lva\lvb}\riA
    \PrInf[\derA_2]
    \PrUn{\fa_1 \land \fa_2}
    \PrLbl{\RuleName\land{E}}
    \PrUn{\fa_i}
    \PrLbl[\riA]\EM
    \PrBin{\fa_i}
    \DisplayProof
    \newline
    \header{\RedNamePerm[\lor]\EM}
    \wholeline{
      \PrAss{\quant\forall\lva\lafa}\riB
      \PrInf[\derB_1]
      \PrUn{\fa_1 \lor \fa_2}
      \PrAss{\lnot\lafa\subst\lva\lvb}\riB
      \PrInf[\derB_2]
      \PrUn{\fa_1 \lor \fa_2}
      \PrLbl[\riB]\EM
      \PrBin{\fa_1 \lor \fa_2}
      \PrAss{\fa_1}\riA
      \PrInf[\derA_1]
      \PrUn\fb
      \PrAss{\fa_2}\riA
      \PrInf[\derA_2]
      \PrUn\fb
      \PrLbl[\riA]{\RuleName\lor{E}}
      \PrTri\fb
      \DisplayProof
    } \\ 
    \nextrule[\downarrow] \\ 
    \wholeline{
      \PrAss{\quant\forall\lva\lafa}\riB
      \PrInf[\derB_1]
      \PrUn{\fa_1 \lor \fa_2}
      \PrAss{\fa_1}\riA
      \PrInf[\derA_1]
      \PrUn\fb
      \PrAss{\fa_2}\riA
      \PrInf[\derA_2]
      \PrUn\fb
      \PrLbl[\riA]{\RuleName\lor{E}}
      \PrTri\fb
      \PrAss{\lnot\lafa\subst\lva\lvb}\riB
      \PrInf[\derB_2]
      \PrUn{\fa_1 \lor \fa_2}
      \PrAss{\fa_1}\riA
      \PrInf[\derA_1]
      \PrUn\fb
      \PrAss{\fa_2}\riA
      \PrInf[\derA_2]
      \PrUn\fb
      \PrLbl[\riA]{\RuleName\lor{E}}
      \PrTri\fb
      \PrLbl[\riB]\EM
      \PrBin\fb
      \DisplayProof
    } 
    \newline
    \header{\RedNamePerm[\limply]\EM}
    \wholeline{
      \PrAss{\quant\forall\lva\lafa}\riA
      \PrInf[\derA_1]
      \PrUn{\fa_1 \limply \fa_2}
      \PrAss{\lnot\lafa\subst\lva\lvb}\riA
      \PrInf[\derA_2]
      \PrUn{\fa_1 \limply \fa_2}
      \PrLbl[\riA]\EM
      \PrBin{\fa_1 \limply \fa_2}
      \PrAx{}
      \PrInf[\derB]
      \PrUn{\fa_1}
      \PrLbl{\RuleName\land{E}}
      \PrBin{\fa_2}
      \DisplayProof
    } \\ 
    \nextrule[\downarrow] \\ 
    \wholeline{
      \PrAss{\quant\forall\lva\lafa}\riA
      \PrInf[\derA_1]
      \PrUn{\fa_1 \limply \fa_2}
      \PrAx{}
      \PrInf[\derB]
      \PrUn{\fa_1}
      \PrLbl{\RuleName\limply{E}}
      \PrBin{\fa_2}
      \PrAss{\lnot\lafa\subst\lva\lvb}\riA
      \PrInf[\derA_2]
      \PrUn{\fa_1 \limply \fa_2}
      \PrAx{}
      \PrInf[\derB]
      \PrUn{\fa_1}
      \PrLbl{\RuleName\limply{E}}
      \PrBin{\fa_2}
      \PrLbl[\riA]\EM
      \PrBin{\fa_2}
      \DisplayProof
    }
  \end{rulelisting}
\end{figure}
\begin{figure}[h!]
  \caption{The permutative reductions of the \(\EM\) rule with the \(\RuleName\forall{E}\) and \(\RuleName\exists{E}\) rules.}
  \label{fig:em_perms_quantifiers}
  \begin{rulelisting}[\leadsto]
    \header{\RedNamePerm[\forall]\EM}
    \PrAss{\quant\forall\lva\lafa}\riA
    \PrInf[\derA_1]
    \PrUn{\quant\forall\lva \fa}
    \PrAss{\lnot\lafa\subst\lva\lvb}\riA
    \PrInf[\derA_2]
    \PrUn{\quant\forall\lva \fa}
    \PrLbl[\riA]\EM
    \PrBin{\quant\forall\lva \fa}
    \PrLbl{\RuleName\forall{E}}
    \PrUn{\fa\subst\lva\lta}
    \DisplayProof
    \nextrule
    \PrAss{\quant\forall\lva\lafa}\riA
    \PrInf[\derA_1]
    \PrUn{\quant\forall\lva \fa}
    \PrLbl{\RuleName\forall{E}}
    \PrUn{\fa\subst\lva\lta}
    \PrAss{\lnot\lafa\subst\lva\lvb}\riA
    \PrInf[\derA_2]
    \PrUn{\quant\forall\lva \fa}
    \PrLbl{\RuleName\forall{E}}
    \PrUn{\fa\subst\lva\lta}
    \PrLbl[\riA]\EM
    \PrBin{\fa\subst\lva\lta}
    \DisplayProof
    \newline
    \header{\RedNamePerm[\forall]\EM}
    \wholeline{
      \PrAss{\quant\forall\lva\lafa}\riA
      \PrInf[\derA_1]
      \PrUn{\quant\exists\lva \fa}
      \PrAss{\lnot\lafa\subst\lva\lvb}\riA
      \PrInf[\derA_2]
      \PrUn{\quant\exists\lva \fa}
      \PrLbl[\riA]\EM
      \PrBin{\quant\exists\lva \fa}
      \PrAss{\fa\subst\lva\lvb}
      \PrInf[\derB]
      \PrUn\fb
      \PrLbl{\RuleName\exists{E}}
      \PrBin\fb
      \DisplayProof
    } \\ 
    \nextrule[\downarrow] \\ 
    \wholeline{
      \PrAss{\quant\forall\lva\lafa}\riA
      \PrInf[\derA_1]
      \PrUn{\quant\exists\lva \fa}
      \PrAss{\fa\subst\lva\lvb}
      \PrInf[\derB]
      \PrUn\fb
      \PrLbl{\RuleName\exists{E}}
      \PrBin\fb
      \PrAss{\lnot\lafa\subst\lva\lvb}\riA
      \PrInf[\derA_2]
      \PrUn{\quant\exists\lva \fa}
      \PrAss{\fa\subst\lva\lvb}\riB
      \PrInf[\derB]
      \PrUn\fb
      \PrLbl[\riB]{\RuleName\exists{E}}
      \PrBin\fb
      \PrLbl[\riA]\EM
      \PrBin\fb
      \DisplayProof
    }
  \end{rulelisting}
\end{figure}

This reduction moves elimination rule instances from ``outside'' or ``below'' to ``inside'' or ``above'' an \(\EM\) rule instance. 
This is useful because an \(\EM\) rule instance may happen in between an introduction rule instance and an elimination rule instance, preventing a proper reduction from taking place.


\section{Witness Extraction}

In this section we prove the witness extraction theorem, that shows how we can extract witnesses from suitable classical derivations in \(\text{HA}+\EM\), as we can do for intuitionistic derivations in HA. 

In order to state and prove our results we need to keep track of free term variables in a derivations, 
since both the \(\RedNameInd\) and the \(\RedNameWitness\) reductions can only be performed when certain terms and formulas are closed. 

We need to define when a variable is free in a derivation. 
\begin{definition}[Free term variables]\label{def:free_term_variable}
  We say that a rule instance \(\riA\) \emph{binds} a term variable that occurs free in the derivation \(\derA\) of a premiss of \(\riA\) in the following cases:
  \begin{itemize}
    \item \(\riA\) is an instance of the \(\RuleName\forall{I}\) rule and binds the variable \(\lva\) in the formula occurrences in the derivation of its premiss: 
      \[
        \PrAx{}
        \PrInf[\derA]
        \PrUn\fa
        \PrLbl[\riA]{\RuleName\forall{I}}
        \PrUn{\quant\forall\lva \fa}
        \DisplayProof
      \]
    \item \(\riA\) is an instance of the \(\RuleName\exists{E}\) rule and binds the variable \(\lvb\) in the formula occurrences in the derivation of its rightmost premiss:
      \[
        \PrAx{\quant\exists\lva \fa}
        \PrAss{\fa\subst\lva\lvb}\riA
        \PrInf[\derA]
        \PrUn\fb
        \PrLbl[\riA]{\RuleName\exists{E}}
        \PrBin\fb
        \DisplayProof
      \]
    \item \(\riA\) is an instance of the \(\RuleNameIndE\) rule and binds the variable \(\lva\) in the formula occurrences in the derivation of its rightmost premiss:
      \[
        \PrAx{\fa\subst\lva{\num 0}}
        \PrAss\fa\riA
        \PrInf[\derA]
        \PrUn{\fa\subst\lva{\succ(\lva)}}
        \PrLbl[\riA]\RuleNameIndE
        \PrBin{\fa\subst\lva\lta}
        \DisplayProof
      \]
    \item \(\riA\) is an instance of the \(\EM\) rule and binds the variable \(\lvb\) in the formula occurrences in the derivation of its rightmost premiss: 
      \[
        \PrAss{\quant\forall\lva \lafa}\riA
        \PrInf
        \PrUn\fb
        \PrAss{\lnot\lafa\subst\lva\lvb}\riA
        \PrInf[\derA]
        \PrUn\fb
        \PrLbl[\riA]\EM
        \PrBin\fb
        \DisplayProof
      \]
  \end{itemize}
  We say that a term variable occurrence is \emph{free in a derivation} when the term variable occurs free in a formula occurrence in the derivation and is not bound by any rule instance. 
  A derivation is \emph{closed} if it has no free term variable nor open assumption. 
\end{definition}

Note that no reduction introduces free term variables in a derivation. 

Since a derivation is a tree, it makes sense to give the definition of branch. 
Principal branches are branches of a derivation that contains only major premisses of elimination and \(\EM\) rule instances. 
\begin{definition}[Principal branch]
  A \emph{branch} in a derivation \(\derA\) is a sequence of formula occurrences \(\foA_0, \dotsc, \foA_n\) in \(\derA\) such that:
  \begin{itemize}
    \item \(\foA_0\) is a top formula occurrence, that is, 
      \(\foA_0\) is either an assumption or the conclusion of an atomic axiom; 
    \item \(\foA_i\) and \(\foA_{i+1}\) are respectively a premiss and the conclusion of the same rule instance \(\riA_{i+1}\), for all \(0 \leq i < n\);
    \item \(\foA_n\) is the conclusion of \(\derA\).
  \end{itemize}
  A branch is \emph{principal} if, for all \(0 \leq i < n\) such that \(\riA_i\) is an elimination or \(\EM\) rule instance, \(\foA_i\) is the major (leftmost) premiss of \(\riA_i\). 
\end{definition}
We use the variables \(\brA, \brB\) for branches. 

In order to study the properties of normal proofs we only need to consider the structure of principal branches. 
A head-cut is the lowest point of a principal branch where a reduction is possible. 
\begin{definition}[Head-cut]
  The \emph{head-cut} of a principal branch \(\brA = \foA_1, \dotsc, \foA_n\) is the formula occurrence \(\foA_i\) with the maximum index \(i\) such that one of the following holds:
  \begin{itemize}
    \item 
      \(\foA_i\) is the conclusion of an elimination rule instance \(\riA_i\), 
      \(\foA_{i-1}\) is the major premiss of \(\riA_i\) and the conclusion of an introduction rule instance \(\riA_{i-1}\);
      when \(\riA_{i-1}\) is a \(\RuleName\land{I}\) rule instance 
      we also require that 
      \(\foA_{i-2}\) is an occurrence of the same formula as \(\foA_i\) (proper reductions);
    \item 
      \(\foA_i\) is the conclusion of an \(\RuleNameIndE\) rule instance \(\riA_i\), 
      whose main term is either \(\num 0\) or \(\succ(\ltb)\) for some term \(\ltb\) (\(\RedNameInd\) reduction); 
    \item 
      \(\foA_i\) is the conclusion of an \(\EM\) rule instance \(\riA\) 
      and either
      \(\foA_{i-1}\) is derived without using the assumption discharged by \(\riA\)
      or
      \(\foA_0\) is an occurrence of the universal assumption discharged by \(\riA\)
      and 
      \(\foA_1\) is the occurrence of a closed atomic formula (\(\RedNameWitness\) reduction);
    \item
      \(\foA_i\) is the conclusion of an elimination rule instance \(\riA\) 
      and 
      \(\foA_{i-1}\) is the conclusion of an \(\EM\) rule instance 
      (\(\RedNamePerm\EM\) reductions). 
  \end{itemize}
  If such an \(i\) exists we say that there is a head-cut along the branch \(\brA\). 
\end{definition}
This definition is the result of a analysis of the conditions that must be met in order to perform one of the reductions we have listed. 
In particular note how, in the condition given for the \(\RedNameWitness\) reduction, the fact that \(\foA_1\) is atomic implies that \(\foA_1\) is the conclusion of a \(\RuleName\forall{E}\) rule instance, as we assumed in defining \(\RedNameWitness\). 

We shall show that, with suitable assumptions,
we can perform the \(\RedNameWitness\) reduction as needed in order to extract a witness from a derivation. 
One of these assumptions is that the conclusion of the derivation is ``simple'' enough, as we define next.

\begin{definition}[Simple Formulas]\label{def:simple_formula}
  We say that a formula is \emph{simply existential} (resp. \emph{universal}) when it is \(\quant\exists\lva \lafa\) (resp. \(\quant\forall\lva \lafa\)) for some atomic formula \(\lafa\).

  We say that a formula is \emph{simple} when it is closed and atomic or simply existential. 
\end{definition}

In the following we consider the \(\EM\) and \(\RuleNameIndE\) rules to be neither elimination nor introduction rules and we give them special treatment. 
 
As we shall show later, 
principal branches beginning with an open assumption have particular structure in normal derivations: they begin with a sequence of elimination rule instances, followed by atomic and \(\EM\) rule instances and they end with introduction and \(\EM\) rule instances. Any of these parts may be missing.
\begin{definition}[Open normal form]
  \label{def:open_normal_form}
  A principal branch \(\foA_0, \dotsc, \foA_n\) is said to be in \emph{open normal form} when
  there exist three natural numbers \(n_E, n_A\) and \(n_I\)  such that \(n_E + n_A n_I = n \) and:
  \begin{itemize}
    \item 
      \(\foA_0\) is the occurrence of an open assumption in \(\Pi\),
    \item 
      \(\foA_i\) is the conclusion of an elimination rule instance 
      for \( 0 < i \leq n_E \), 
    \item 
      \(\foA_i\) is the conclusion of an atomic or \(\EM\) rule instance 
      for \( n_E < i \leq n_E + n_A \), 
    \item 
      \(\foA_{n_E + n_A + 1}\) is the conclusion of an introduction rule instance\footnotemark, 
    \item 
      \(\foA_i\) is the conclusion of an introduction or \(\EM\) rule instance 
      for \( n_E + n_A < i \leq n \), 
  \end{itemize}
  \(n_E, n_A\) and \(n_I\) are the number of elimination, atomic or \(\EM\), introduction or \(\EM\) rule instances, respectively.
  \footnotetext{Since \(\EM\) rule instances can appear intermingled with both atomic and introduction rule instances, in the definition we require that \(\foA_{n_E + n_A + 1}\) be the conclusion of an introduction rule, so that \(n_A\) and  \(n_I\) are uniquely determined. } 
\end{definition}

We can now prove our main result: 
closed normal derivations of simply existential formulas in \(\text{HA}+\EM\) can be reduced to derivations ending with an introduction rule instance. 
Derivations in HA have a similar property. 
The theorem we are going to prove holds for derivations that are concrete enough, namely they are:
self-contained (without open assumptions), concrete (without open term variables) and with an effective conclusion (a simply existential formula). 
The proof is split into several lemmas, whose proofs are given in \Cref{sec:proofs}.

In the first lemma we show that, in a derivation of a simply existential with no free term variables, a simply universal assumption is followed by a closed atomic formula. 
This will be used later to prove that we can perform the \(\RedNameWitness\) reduction on universal assumption of an \(\EM\) rule instance. 
\begin{lemma}\label{thm:forall_elim_closed}
  Let \(\brA = \foA_0, \dotsc, \foA_n\) be a principal branch in open normal form in a derivation \(\derA\) in \(\text{HA}+\EM\), with \(n_E, n_A\) and \(n_I\) defined as in \Cref{def:open_normal_form}.
  Let \(\fa_0, \dotsc, \fa_n\) be the formulas \(\foA_0, \dotsc, \foA_n\) are occurrences of. 
  Then the following statements hold: 
  \begin{enumerate}
    \item\label{thm:subformula_intro_branch_a}
      \(\fa_i\) is a non-atomic subformula of \(\fa_n\) for all \(n_E + n_A < i \leq n\);
    \item\label{thm:subformula_intro_branch}
      if some \(\foA_i\) is the conclusion of an introduction rule instance, 
      then \(\fa_i\) is a subformula of \(\fa_n\);
  \end{enumerate}
  Moreover assume that \(\fa_n\) is a simple formula. Then:
  \begin{enumerate}[resume]
    \item\label{thm:free_var_branch}
      if a term variable \(\lva\) is free in some \(\fa_i\), 
      then \(\lva\) is free in \(\derA\);
    \item\label{thm:forall_elim_branch}
      if some \(\fa_i\) is simply universal,
      then \(\foA_i\) is the premiss of a \(\RuleName\forall{E}\) rule instance;
    \item\label{thm:cut_branch}
      if \(\derA\) has no free term variables and \(\fa_i\) is simply universal,
      then \(\fa_{i+1}\) is a closed atomic formula. 
  \end{enumerate}
\end{lemma}

In the following lemma we show how we can apply the \(\RedNameWitness\) reduction. 
\begin{lemma}[\(\EM\) reduction]\label{lemma:em_open_branch}
  Let \(\derA\) be a derivation in \(\text{HA}+\EM\) with no free term variables. 
  Assume that \(\derA\) ends with an \(\EM\) rule instance \(\riA\) whose conclusion is an occurrence of a simple formula. 
  Assume that the derivation \(\derA'\) of the leftmost premiss of \(\riA\) has a principal branch \(\brA\) in open normal form.
  Then at least one of the following occurs:
  \begin{enumerate}
    \item \(\derA\) has a head-cut or a non-normal term along a principal branch,
    \item \(\derA\) has a principal branch in open normal form.
  \end{enumerate}
\end{lemma}
\begin{proof} 
  Let \(\brA = \foA_0, \dotsc, \foA_n\) and let \(\fa_0, \dotsc, \fa_n\) be the formulas \(\foA_0, \dotsc, \foA_n\) are occurrences of. 
  Let \(\foA\) be the conclusion of \(\derA\) and of the \(\EM\) rule instance \(\riA\):
  \[
    \PrAss[\foA_0]{\fa_0}{}
    \PrInfBr{\derA'}\brA
    \PrUn[\foA_n]{\fa_n}
    \PrAx{}
    \PrInf
    \PrUn{\fa_n}
    \PrLbl[\riA]\EM
    \PrBin[\foA]{\fa_n}
    \DisplayProof
  \]
  Note that we can extend \(\brA\) to \(\brB = \foA_0, \dotsc, \foA_n, \foA\) and \(\brB\) is a principal branch of \(\derA\).

  If \(\foA_0\) is not discharged by \(\riA\), then \(\brB\) is a principal branch in open normal form and thus we get the statement. 
  Otherwise, \(\foA_0\) is discharged by \(\riA\), meaning that \(\fa_0\) is the universal assumption of the \(\EM\) instance. 
  We can apply (\ref{thm:cut_branch}) of \Cref{thm:forall_elim_closed} to \(\brB\), since \(\foA_0\) is simply universal, \(\foA\) is simply existential and \(\derA\) has no free term variables. 

  Then \(\foA_1\) is a closed atomic formula and 
  we can perform the \(\RedNameWitness\) reduction, that is,
  there is a head-cut along the principal branch \(\brB\) of \(\derA\) and we can conclude. \qed

\end{proof}

The following lemma shows how to handle \(\RuleNameIndE\) rule instances. 
\begin{lemma}[Induction normalization]\label{thm:induction}
  Let \(\derA\) be a derivation in \(\mathrm{HA}+\EM\) ending with an \(\RuleNameIndE\) rule instance. 
  Then at least one of the following holds:
  \begin{enumerate}
    \item \(\derA\) has a head-cut or a non-normal term along a principal branch, 
    \item \(\derA\) contains a free term variable. 
  \end{enumerate}
\end{lemma}

The following lemma can be thought of as a weak result on the structure of derivations. 
\begin{lemma}[Structure of Normal Form]\label{thm:em_reduction}
  Let \(\derA\) be a derivation in \(\mathrm{HA}+\EM\). Then at least one of the following holds:
  \begin{enumerate}
    \item \(\derA\) has a head-cut or a non-normal term along a principal branch; 
    \item \(\derA\) contains a free term variable; 
    \item \(\derA\) has a principal branch in open normal form; 
    \item \(\derA\) ends with an introduction rule instance;
    \item \(\derA\) is atomic (only atomic formulas occur in \(\derA\)); 
    \item \(\derA\) ends with an \(\EM\) instance and its conclusion is not simple.
  \end{enumerate}
\end{lemma}

Our main theorem is now an easy corollary of the previous lemma.
\begin{theorem}[Witness Extraction] \label{thm:witness_extraction}
  Let \(\derA\) be a derivation of a simple formula \(\fa\) in \(\mathrm{HA}+\EM\). Assume that:
  \begin{enumerate}
    \item \(\derA\) has no principal branch with a head-cut or a non-normal term; 
    \item \(\derA\) contains no free term variable; 
    \item \(\derA\) has no open assumptions;
  \end{enumerate}
  Then \(\derA\) is either atomic or ends with a \(\RuleName\exists{I}\) rule instance. 
  In particular, if \(\derA\) is closed, normal and \(\fa\) is simply existential then \(\derA\) ends with an introduction. 
\end{theorem}

Note that we did not prove that the normalization of a derivation is convergent: this shall be the focus of a future paper. 
For the moment we just assume that the reduction of a derivation must stop after a finite number of steps and produces a derivation without head-cuts.
Under this assumption, 
\Cref{thm:witness_extraction} shows that our proof reduction can extract a witness from the derivation of a closed formula \(\quant\exists\lva \lafa\), which can be found in the premise of the \(\RuleName\exists{I}\) rule instance at the end of the normalized derivation, by the definition of the \(\RuleName\exists{I}\) rule.

\bibliographystyle{plain} 
\bibliography{../realizability} 

\appendix
\section{Appendix}\label{sec:appendix}

\subsection{Formula occurrences and rule instances}\label{app:occurrences_instances}
A formula can occur more than once in a derivation. 
While these occurrences are clearly distinct in a tree-like representation, 
in order to avoid confusion when referring to them in the text special care must be taken. 
Thus we make a distinction between formulas and \emph{formula occurrences}, or simply occurrences, which we label with \(\foA, \foB, \foC\). 
In a derivation, formula occurrences are arranged following the patterns given by the inference rules. 
As with formulas, we distinguish between rules and \emph{rule instances}, or simply instances, which we label with \(\riA,\riB, \riC\). 
We write a derivation \(\derA\) as follows:
\[
  \PrAx{}
  \PrInf[\derA_1]
  \PrUn[\foB_1]\fb
  \PrAss\fc\riB
  \PrInf[\derA_2]
  \PrUn[\foB_2]\fb
  \PrLbl[\riA]{\text{rulename}}
  \PrBin[\foA]\fa
  \DisplayProof 
\]
The only occurrence of the formula \(\fa\) is labeled \(\foA\), while \(\fb\) occurs two times, with distinct labels \(\foB_1\) and \(\foB_2\). 
\(\foA\) is the \emph{conclusion} of an instance, labeled \(\riA\), of an inference rule named \(\text{rulename}\). 
\(\foA_1\) and \(\foA_2\) are the \emph{premisses} of \(\riA\). 
We also say that \(\foA\) is the \emph{conclusion} of the whole derivation \(\derA\). 
With \(\derA_1\) and \(\derA_2\) we denote two \emph{subderivations} (as in subtree) of \(\derA\). 
We distinguish subderivations by their conclusion, so we say that \(\derA_i\) is the (sub)derivation of \(\foA_i\) for \(i=1,2\).
By writing \(\PrAss\fc\riB \DisplayProof\) in square brackets above \(\derA_2\), 
we make explicit that \(\derA_2\) may contain occurrences of the assumption \(\fc\), which is discharged by some undisplayed rule instance \(\riB\).


\subsection{Equivalence between the \(\EM\) axiom and the \(\EM\) rule} \label{sec:em_axiom_rule_equiv}
We show how to derive the \(\EM\) rule from the \(\EM\) axiom and vice versa.
The \(\EM\) rule is derived by an \(\RuleName\lor{E}\) rule instance, whose major premiss is an instance of the \(\EM\) axiom and whose rightmost assumption is the major premiss of an \(\RuleName\exists{E}\) instance: 
\[
  \PrAx{}
  \PrLbl\EM
  \PrUn{(\quant\forall\lva \fa) \lor (\quant\exists\lva \lnot\fa)}
  \PrAss{\quant\forall\lva \fa}\riA
  \PrInf
  \PrUn\fc
  \PrAss{\quant\exists\lva \lnot\fa}\riA
  \PrAss{\lnot\fa\subst\lva\lvb}\riB
  \PrInf
  \PrUn\fc
  \PrLbl[\riB]{\RuleName\exists{E}}
  \PrBin\fc
  \PrLbl[\riA]{\RuleName\lor{E}}
  \PrTri\fc
  \DisplayProof
\]
On the other hand, the \(\EM\) axiom can be derived from the \(\EM\) rule by two \(\RuleName\lor{I}\) instances:
\[
  \PrAss{\quant\forall\lva \fa}\riA
  \PrLbl{\RuleName\lor{I}}
  \PrUn{(\quant\forall\lva \fa) \lor (\quant\exists\lva \lnot \fa)}
  \PrAss{\lnot\fa\subst\lva\lvb}\riA
  \PrLbl{\RuleName\exists{I}}
  \PrUn{\quant\exists\lva \lnot \fa}
  \PrLbl{\RuleName\lor{I}}
  \PrUn{(\quant\forall\lva \fa) \lor (\quant\exists\lva \lnot \fa)}
  \PrLbl[\riA]\EM
  \PrBin{(\quant\forall\lva \fa) \lor (\quant\exists\lva \lnot \fa)}
  \DisplayProof
\]
\subsection{Proofs}\label{sec:proofs}
In this subsection we list the proofs of the lemmas. 

Proof of \Cref{thm:forall_elim_closed}.
\begin{proof} 
  (\ref{thm:subformula_intro_branch}) follows immediately from  (\ref{thm:subformula_intro_branch_a}). 
  We need (\ref{thm:subformula_intro_branch}) to prove (\ref{thm:free_var_branch}) and (\ref{thm:forall_elim_branch}). 
  Then, by (\ref{thm:free_var_branch}) and (\ref{thm:forall_elim_branch}), we prove (\ref{thm:cut_branch}). 
  Here are the proofs. 
  \begin{enumerate}
    \item
      We proceed by induction on \(n_I\). 
      \begin{itemize}
        \item 
          If \(n_I = 0\), the thesis holds vacuously. 
        \item
          If \(n_I = 1\), we need to prove the statement just for \(i = n_E + n_A + 1 = n\) and thus \(\foA_n\) is the conclusion of an introduction rule instance by \Cref{def:open_normal_form}. 
          This means that \(\fa_n\) is not atomic. 
          Obviously it is also a subformula of itself so we are done. 
        \item
          Otherwise, let \(n_I > 1 \). 

          Consider the subderivation \(\derA'\) of \(\derA\) ending with \(\foA_{n-1}\) and its principal branch \(\brA' = \foA_0, \dotsc, \foA_{n - 1}\). 
          \(\brA'\) is in open normal form in \(\derA'\), with \(n_E' = n_E, n_A' = n_A\) and \(n_I' = n_I - 1\).
          Then, by inductive hypothesis, 
          for all \(n_E + n_A < i \leq n-1\),
          \(\fa_i\) is a non-atomic subformula of \(\fa_{n-1}\). 

          By \Cref{def:open_normal_form},
          \(\derA\) ends with an introduction or \(\EM\) rule instance \(\riA\). 
          In both cases \(\fa_{n-1}\) is a subformula of \(\fa_n\), since \(\foA_{n-1}\) is the premiss of \(\riA\) and \(\foA_n\) is its conclusion. 

          Thus for all \(n_E + n_A < i \leq n\), \(\fa_i\) is a subformula of \(\fa_n\). 
          Moreover since \(\fa_{n-1}\) is non-atomic then \(\fa_n\) is too. 
      \end{itemize}
    \item 
      If \(\foA_i\) is the conclusion of an introduction rule instance then \(n_E + n_A < i \leq n \) by \Cref{def:open_normal_form}. 
      Thus we conclude by (\ref{thm:subformula_intro_branch_a}). 
    \item 
      We show that \(\lva\) is not bound by any rule instance and thus is free in \(\derA\).
      The only rule that binds a variable above a major premiss, 
      and thus the only rule that can bind a variable in a principal branch, 
      is the \(\RuleName\forall{I}\) rule.
      Now assume that a \(\RuleName\forall{I}\) rule instance occur along \(\brA\) with conclusion \(\foA_j\). 
      By the previous statement (\ref{thm:subformula_intro_branch}), 
      \(\fa_j\) is a subformula of \(\fa_n\). 
      This yields a contradiction because by assumption \(\fa_n\) is simple and \(\fa_j\) is universally quantified since \(\foA_j\) is the conclusion of a \(\RuleName\forall{I}\) rule instance. 
    \item 
      We show that \(\foA_i\) is the premiss of a \(\RuleName\forall{E}\) rule instance because no other alternative is possible. 
      \begin{itemize}
        \item 
          \(\foA_i\) cannot be the premiss of an atomic rule instance, since we assumed that \(\fa_i\) is simply universal and thus not atomic.
        \item 
          \(\foA_i\) cannot be the premiss of an introduction rule instance, since in that case \(\fa_{i+1}\) is a subformula of \(\fa_n\) by (\ref{thm:subformula_intro_branch}). 
          Therefore a simply universal formula \(\fa_i\) is a subformula of a simple formula \(\fa_n\), which is a contradiction. 
        \item 
          Finally \(\foA_i\) cannot be the premiss of an \(\EM\) rule instance. 
          More precisely assume that \(\foA_i\) is followed by exactly \(j > 0\) instances of the \(\EM\) rule. 
          Then \(\foA_{i+j}\) is the conclusion of the last \(\EM\) rule instance \(\riA\) and \(\fa_{i+j}\) is the same formula as \(\fa_i\), in particular \(\fa_{i+j}\) is simply universal. 
          By definition of open normal form, \(\EM\) rule instances can only be followed by introduction, atomic or \(\EM\) rule instances. 
          Since we assumed that there are exactly \(j\) instances of the \(\EM\) rule, \(\foA_{i+j}\) is the premiss of either an atomic or introduction rule instance. 
          Then we are in one of the previous cases and we have a contradiction.
      \end{itemize}
      Then \(\foA_i\) can only be the premiss of an elimination rule instance and, 
      being \(\fa_i\) simply universal, 
      it must be an instance of the \(\RuleName\forall{E}\) rule. 
    \item 
      By (\ref{thm:forall_elim_branch}) we known that \(\foA_{i+1}\) is the conclusion of a \(\RuleName\forall{E}\) rule instance whose premiss is simply universal.
      Therefore \(\fa_{i+1}\) is an atomic formula.
      If \(\fa_{i+1}\) has a free term variable, \(\derA\) has too by (\ref{thm:free_var_branch}). 
      Since we assumed that \(\derA\) has no free term variable, \(\fa_{i+1}\) must be closed.
      Therefore \(\fa_{i+1}\) is a closed atomic formula. \qed
  \end{enumerate}
\end{proof}

Proof of \Cref{thm:induction}.
\begin{proof}
  Let \(\riA\) be the \(\RuleNameIndE\) rule instance \(\derA\) ends with
  and
  let its conclusion be an occurrence \(\foA\) of some formula \(\fa\subst\lva\lta\).
  If \(\foA\) has a free term variable \(\lva\) then \(\lva\) is free in \(\derA\) too and we are done. 
  If \(\lta\) is not normal then all principal branches\footnote{Since all branches of \(\derA\) end with its conclusion.} in \(\derA\) have a non-normal term. 
  Otherwise \(\lta\) is a closed normal term and thus 
  it is either \(\num 0\) or \( \succ(\ltb)\) for some term \(\ltb\) and e we can apply the \(\RedNameInd\) reduction, meaning that any principal branch of \(\derA\) has a head cut. \qed
\end{proof}

Proof of \Cref{thm:em_reduction}.
\begin{proof} 
  The proof is by induction on the structure of the derivation \(\derA\), that is, we assume that the statement holds for all subderivations of \(\derA\) and we prove that it holds for the whole derivation. 

  Let \(\riA\) be the last rule instance in \(\derA\). 
  If \(\riA\) in an introduction rule instance the statement is satisfied and we are done. 

  If \(\riA\) is an \(\RuleNameIndE\) rule instance then we get the statement by applying \Cref{thm:induction} to \(\derA\).

  Then we only need to understand what happens when \(\riA\) is an elimination, an atomic or an \(\EM\) rule instance. 
  Note that the only case in which \(\riA\) has no premisses is when \(\riA\) is an atomic axiom.
  If this happens then \(\derA\) is atomic (it is just the conclusion of \(\riA\)) and the statement is satisfied.

  Otherwise \(\riA\) has one or more (when \(\riA\) is atomic) major premisses. 
  Let \(\derA'\) be the derivation of any one of the major premisses of \(\riA\).
  Any principal branch \(\brA\) of \(\derA'\) can be extended to a branch \(\brB\) of \(\derA\), by appending the conclusion of \(\derA\). 
  \(\brB\) is principal too because \(\derA'\) is the subderivation of a major premiss of \(\riA\). 
  We shall use this fact often in the following. 


  By inductive hypothesis \(\derA'\) satisfies the statement, 
  so we proceed by considering all the possible cases. 
  \begin{enumerate}
    \item\label{case:cut} 
      \(\derA'\) has  a head-cut or a non-normal term along a principal branch \(\brA\). 
      As we noted \(\brA\) can be extended to a principal branch of \(\derA\) with the same head-cut or non-normal term,
      so \(\derA\) satisfies the statement and we are done. 
    \item\label{case:free_var} 
      \(\derA'\) contains a free term variable. 

      There are four rules that can bind term variables: the \(\RuleName\forall{I}, \RuleName\exists{E}, \RuleNameIndE\) and \(\EM\) rules. 
      Since the cases when \(\riA\) is an introduction or \(\RuleNameIndE\) rule instance have been taken care of already and 
      since the \(\RuleName\exists{E}\) and \(\EM\) rules can only bind term variables in the derivation of its minor premiss, 
      any free term variable in \(\derA'\) is free in \(\derA\) too. 
      Thus \(\derA\) satisfies the statement. 

    \item\label{case:open_ass} 
      \(\derA'\) has a principal branch \(\brA = \foA_0, \dotsc, \foA_n\) in open normal form. 
      Let \(\brB\) be the principal branch of \(\derA\) extending \(\brA\). 
      \[
        \PrAss[\foA_0]{\fa_0}{}
        \PrInfBr{\derA'}\brA
        \PrUn[\foA_n]{\fa_n}{}
        \PrAx\dotso
        \PrLbl[\riA]{\EM|\RuleName\ast{E}|\mathcal{A}}
        \PrBin\fa
        \DisplayProof
      \]
      Note that elimination rule instances do not discharge assumptions in their leftmost subderivation and atomic rule instances do not discharge assumptions in general. 

      Then, when \(\riA\) is either an elimination or atomic rule instance, 
      the assumption \(\foA_0\) is still open in \(\derA\) and 
      we have the following cases depending on how which rule \(\foA_n\) is the conclusion of. 
      Note that since \(\brA\) is in open normal form, \(\foA_n\) cannot be an \(\RuleNameIndE\) instance. Thus we have the following cases: 
      \begin{center}
        \begin{tabular}{cc|c|c|c|c|c|}
          \cline{3-6}
          & &\multicolumn{4}{|c|}{\(\foA_n\) is the conclusion of} \\ 
          \cline{3-6}
          & & ELIM & ATOM & INTRO & \(\EM\) \\ 
          \hline
          \multicolumn{1}{|c|}{\multirow{2}{*}{\(\riA\)}} & ELIM & EXT & NO & CUT & PERM \\
          \cline{2-6}
          \multicolumn{1}{|c|}{} & ATOM & EXT & EXT & NO & NO/EXT \\
          \hline
        \end{tabular}
      \end{center}

      \begin{itemize}
        \item[EXT] 
          \(\brB\) begins with the open assumption \(\foA_0\) followed by elimination and atomic rule instances, so \(\derA\) satisfies the statement; 
        \item[NO] 
          this is never the case since, in a principal branch, 
          an elimination (risp. atomic) rule instance cannot follow an atomic (risp. introduction) rule instance, because
          the major premiss (risp. conclusion) of an elimination (risp. introduction) rule instance is not atomic and thus cannot be the conclusion (risp. premiss) of an atomic rule instance;
        \item[CUT] 
          \(\riA\) is an elimination rule instance and its major premiss is the conclusion of an introduction rule instance, 
          thus \(\brB\) ends with a head-cut and again \(\derA\) satisfies the statement;
        \item[PERM] 
          when a major premiss of \(\riA\) is the conclusion of an \(\EM\) rule instance we can apply the \(\RedNamePerm\EM\) reduction, thus \(\brB\) ends with a head-cut and \(\derA\) satisfies the statement;
        \item[NO/EXT] 
          we have two cases depending on the \(n_I\) of \(\brA\):
          \begin{itemize}
            \item[\(n_I > 0\)] 
              then, by (\ref{thm:subformula_intro_branch_a}) of \Cref{thm:forall_elim_closed}, we have that \(\foA_n\) is an occurrence of a non atomic formula and thus \(\riA\) cannot be an atomic rule instance;
            \item[\(n_I = 0\)]
              in this case \(\brB\) is in open normal form and thus \(\derA\) satisfies the statement.
          \end{itemize}


      \end{itemize}
      On the other hand, if \(\riA\) is an \(\EM\) rule instance then we can apply \Cref{lemma:em_open_branch} to \(\derA\) (since we can safely assume that \(\derA\) contains no free term variables) and we get the statement. 
    \item \(\derA'\) ends with an introduction rule instance \(\riB\).
      Then the conclusion \(\foB\) of \(\riB\) cannot be atomic and since it is a premiss of \(\riA\), \(\riA\) cannot be atomic either. 
      Therefore \(\riA\) must be either an elimination or an \(\EM\) rule instance.

      If \(\riA\) is an elimination then there is a head-cut along a principal branch going through \(\foB\) so \(\derA\) satisfies the statement. 
      \[
        \PrAx{}
        \PrInf
        \PrLbl[\riB]{\RuleName{?}{I}}
        \PrUn[\foB]\fb
        \PrAx\dotso
        \PrLbl[\riA]{\RuleName{?}{E}}
        \PrBin[\foA]\fa
        \DisplayProof
      \]
      If \(\riA\) is an \(\EM\) rule instance then \(\foA\) and \(\foB\) are both occurrences of the formula \(\fa\).
      If \(\fa\) is not simple then \(\derA\) satisfies the statement. 
      Otherwise we assume that \(\fa\) is simple: 
      since \(\fa\) cannot be atomic 
      (it occurs as the conclusion of the introduction rule instance \(\riB\))
      \(\fa\) must be \(\quant\exists\lva \lafa\) for some atomic formula \(\lafa\)
      and
      \(\riB\) must be an \(\RuleName\exists{I}\) rule instance. 
      Then we are in the following situation:
      \[ 
        \PrAx{}
        \PrInf[\derA'']
        \PrUn\fa
        \PrLbl[\riB]{\RuleName\exists{I}}
        \PrUn[\foB]{\quant\exists\lva\lafa}
        \PrAx{}
        \PrInf
        \PrUn{\quant\exists\lva\lafa}
        \PrLbl[\riA]\EM
        \PrBin[\foA]{\quant\exists\lva\lafa}
        \DisplayProof
      \]
      where \(\derA''\) is the derivation of the premiss of \(\riB\). 
      Again note that principal branches of \(\derA''\) extend to principal branches of \(\derA'\) by appending \(\foB\). 

      By inductive hypothesis \(\derA''\) satisfies the statement, 
      so we proceed by considering all the possible cases. 
      \begin{enumerate}
        \item 
          \(\derA''\) has a head-cut or a non-normal term along a principal branch.
          Then \(\derA'\) does too and we are in the previously solved case labeled \ref{case:cut}. 
        \item \(\derA''\) contains a free term variable. 
          Since \(\RuleName\exists{I}\) does not bind free term variables, 
          \(\derA'\) does too and we are in the previously solved case labeled \ref{case:free_var}. 
        \item
          \(\derA''\) has a principal branch \(\brA\) in open normal form. 
          Since \(\RuleName\exists{I}\) does not discharge open assumptions,
          \(\derA'\) does too and we are in the previously solved case labeled \ref{case:open_ass}. 
        \item 
          \(\derA''\) ends with an introduction rule instance. 
          This cannot happen because \(\foB\) is an atomic formula occurrence. 
        \item 
          \(\derA''\) is atomic. 
          The assumption discharged by \(\riA\) from its leftmost subderivation is not atomic, 
          thus it cannot occur in \(\derA''\) since \(\derA''\) is atomic. 
          Therefore we can apply the \(\RedNameWitness\) reduction meaning that 
          there is a head-cut at the end of the principal branches of \(\derA'\)
          and we are in the previously solved case labeled \ref{case:cut}. 
        \item \(\derA''\) ends with an \(\EM\) instance and its conclusion is not simple.
          This cannot happen because we assumed that \(\fa\) is simple and \(\foB\) is an occurrence of \(\fa\).
      \end{enumerate}

    \item \(\derA'\) is atomic. 
      In this case one of major premisses of \(\riA\) is atomic, so \(\riA\) cannot be an elimination rule instance and must be either an \(\EM\) or atomic rule instance. 
      \begin{itemize}
        \item
          If \(\riA\) is an \(\EM\) rule instance then it is redundant: 
          the assumption discharged by \(\riA\) from its leftmost subderivation is not atomic, 
          thus it cannot occur in \(\derA'\) since \(\derA'\) is atomic whose premisses are atomic formula occurrences. 
          Therefore we can apply the \(\RedNameWitness\) reduction to \(\riA\), meaning that 
          there is a head-cut at the end of the principal branches of \(\derA\) and thus \(\derA\) satisfies the statement. 
        \item
          Otherwise, if \(\riA\) is an atomic rule instance, consider the other subderivations of its major premisses. 
          If they are all atomic then \(\derA\) is atomic too and it satisfies the statement. 
          Otherwise there is a major premiss of \(\riA\) with a non atomic derivation \(\derA''\). 
          Then one of the other cases applies with \(\derA''\) in place of \(\derA'\). 
      \end{itemize}

    \item \(\derA'\) ends with an \(\EM\) rule instance \(\riB\) and its conclusion is not simple.
      \(\riA\) cannot be an atomic rule instance since one of its premisses is the conclusion of \(\riB\) which is not simple and thus not atomic. 
      If \(\riA\) is an elimination rule instance we can apply the \(\RedNamePerm\EM\) reduction to \(\riA\) and \(\riB\). 
      Thus there is a head-cut at the end of the principal branches of \(\derA\), 
      and \(\derA\) satisfies the statement. 
      Otherwise, if \(\riA\) is an \(\EM\) rule instance then the conclusions of \(\derA\) and \(\derA'\) are occurrences of the same non-simple formula. 
      Therefore \(\derA\) again satisfies the statement. 
  \end{enumerate}
  Since we exhausted all the possible cases we are done. \qed
\end{proof}

Proof of \Cref{thm:witness_extraction}. 
\begin{proof}
  The hypotheses rule out most of the cases considered by \Cref{thm:em_reduction}. 
  The only possible cases are:
  \begin{enumerate}
    \item \(\derA\) ends with an introduction rule instance,
    \item \(\derA\) is atomic.
  \end{enumerate}
  Since \(\fa\) is simple it is either an atomic or existentially quantified formula. 
  If \(\fa\) is atomic, \(\derA\) cannot end with an introduction rule instance and thus \(\derA\) must be atomic.
  Otherwise, if \(\fa\) is existentially quantified, \(\derA\) cannot end with an atomic rule instance and thus 
  \(\derA\) must end with an introduction rule instance which can only be  a \(\RuleName\exists{I}\) rule instance. \qed
\end{proof}

\begin{thesis}

  \subsection{Permutative reductions}\label{sec:permutative_reductions} 
  We saw how the proper reductions are performed when the conclusion of an introduction rule instance is the major premiss of a elimination rule instance. 
  The situation becomes less straightforward when the conclusion of an introduction rule instance \(\riA\) is a minor premiss of an \(\RuleName\lor{E}\) or \(\RuleName\exists{E}\) rule instance \(\riC\) whose conclusion is in turn the major premiss of an elimination rule instance \(\riB\). 
  Also in this case the formula introduced by \(\riA\) is eliminated by \(\riB\), but we cannot apply a proper reduction since \(\riC\) is in the way. 
  What we can do is to rearrange the derivation by moving \(\riB\) above (or ``inside'') \(\riC\), so that \(\riB\) is immediately below \(\riA\) and we can apply the suitable proper reduction. 
  Therefore we have two \emph{permutative reductions}, depending on whether \(\riC\) is an instance of the \(\RuleName\lor{E}\) or the \(\RuleName\exists{E}\) rule.
  Note that repeated application of the permutative reductions allows us to apply a proper reduction even when there is more than one instance of the \(\RuleName\lor{E}\) or the \(\RuleName\exists{E}\) between \(\riA\) and \(\riB\).  
  Thus they can be thought of as auxiliary reductions that can eventually enable a suitable proper reduction. 
  They are listed in \Cref{fig:permutative_reductions}. 

  \begin{figure}[h]
    \caption{The permutative reductions.}
    \label{fig:permutative_reductions}
    \begin{rulelisting}[]
      \header{\RedNamePerm\lor}
      &
      \PrAx{}
      \PrInf[\derA_1]
      \PrUn{\fa_1 \lor \fa_2}
      \PrAss{\fa_1}\riC
      \PrInf[\derA_2]
      \PrUn\fb
      \PrAss{\fa_2}\riC
      \PrInf[\derA_3]
      \PrUn\fb
      \PrLbl[\riC]{\RuleName\lor{E}}
      \PrTri\fb
      \PrAx{}
      \PrInf[\derA_4]
      \PrLbl[\riB]{\RuleName\ast{E}}
      \PrBin\fc
      \DisplayProof 
      &
      \\
      & \downarrow &
      \\
      &
      \PrAx{}
      \PrInf[\derA_1]
      \PrUn{\fa_1 \lor \fa_2}
      \PrAss{\fa_1}\riC
      \PrInf[\derA_2]
      \PrUn\fb
      \PrAx{}
      \PrInf[\derA_4]
      \PrLbl[\riB]{\RuleName\ast{E}}
      \PrBin\fc
      \PrAss{\fa_2}\riC
      \PrInf[\derA_3]
      \PrUn\fb
      \PrAx{}
      \PrInf[\derA_4]
      \PrLbl[\riB]{\RuleName\ast{E}}
      \PrBin\fc
      \PrLbl[\riC]{\RuleName\lor{E}}
      \PrTri\fc
      \DisplayProof
      &
      \newline
      \header{\RedNamePerm\exists} 
      &
      \PrAx{}
      \PrInf[\derA_1]
      \PrUn{\quant\exists\lva\fa}
      \PrAss{\fa\subst\lva\lvb}\riC
      \PrInf[\derA_2]
      \PrUn\fb
      \PrLbl[\riC]{\RuleName\exists{E}}
      \PrBin\fb
      \PrAx{}
      \PrInf[\derA_4]
      \PrLbl[\riB]{\RuleName\ast{E}}
      \PrBin\fc
      \DisplayProof 
      &
      \\ & \downarrow & \\
         &
      \PrAx{}
      \PrInf[\derA_1]
      \PrUn{\quant\exists\lva\fa}
      \PrAss{\fa\subst\lva\lvb}\riC
      \PrInf[\derA_2]
      \PrUn\fb
      \PrAx{}
      \PrInf[\derA_4]
      \PrLbl[\riB]{\RuleName\ast{E}}
      \PrBin\fc
      \PrLbl[\riC]{\RuleName\exists{E}}
      \PrBin\fc
      \DisplayProof
      &
      \newline
      &
      \begin{tabular}{l}
        \(\ast\) is any logical connective: \(\land, \lor, \limply, \forall, \exists\); \\
        \(\derA_4\) stands for the subderivations of the minor premisses (if any) of \(\riB\).
      \end{tabular}
      &
    \end{rulelisting}

  \end{figure}

  \subsection{Immediate simplifications} 
  Consider a different kind of avoidable complexity in derivation: 
  an instance \(\riA\) of the \(\RuleName\lor{E}\) or the \(\RuleName\exists{E}\) rule such that the one of its minor premisses \(\foA\) is derived without using the assumption discharged by \(\riA\). 
  More precisely this means that no occurrence of the assumption discharged by \(\riA\) appears in the subderivation of \(\foA\). 
  Whenever this is the case we say that \(\riA\) is \emph{redundant} since we do not need the assumptions it provides in order to prove its conclusion. 
  There are two reductions, called \emph{immediate simplifications}, depending on whether \(\riA\) is an instance of the \(\RuleName\lor{E}\) or the \(\RuleName\exists{E}\) rule. 
  They are listed in \Cref{fig:immediate_simplifications}. 

  \begin{figure}[h]
    \caption{The immediate simplifications.}
    \label{fig:immediate_simplifications}
    \begin{rulelisting}[\leadsto]
      \header{\RedNameSimpl\lor}
      \PrAx{\fa_1 \lor \fa_2}
      \PrAss{\fa_1}\riA
      \PrInf[\derA_1]
      \PrUn\fc
      \PrAss{\fa_2}\riA
      \PrInf[\derA_2]
      \PrUn\fc
      \PrLbl[\riA]{\RuleName\lor{E}}
      \PrTri\fc
      \DisplayProof
      \nextrule
      \PrAx{}
      \PrInf[\derA_i]
      \PrUn\fc
      \DisplayProof 
      \\
      \header{\text{ if \(\derA_i\) does not depend on \(\fa_i\)}}
      \hline
      \header{\RedNameSimpl\exists} 
      \PrAx{\quant\exists\lva \fa}
      \PrAss{\fa\subst\lva\lvb}\riA
      \PrInf[\derA]
      \PrUn\fb
      \PrLbl[\riA]{\RuleName\exists{E}}
      \PrBin\fb
      \DisplayProof
      \nextrule
      \PrAx{}
      \PrInf[\derA]
      \PrUn\fc
      \DisplayProof 
    \end{rulelisting}
  \end{figure}

  \subsection{Normal derivations}
  When a derivation cannot be reduced by any of the reductions we listed, it is said to be a \emph{normal} derivation. 
  If we make the ``cannot be reduced'' part explicit, we get the following definition. 
  \begin{definition}
    A derivation is \emph{directly reducible} when one of the following holds: 
    \begin{itemize}
      \item it ends with an elimination rule instance whose major premiss is the conclusion of an introduction rule instance (proper reduction),
      \item it ends with an elimination rule instance whose major premiss is the conclusion of an \(\RuleName\lor{E}\) or \(\RuleName\lor{E}\) rule instance (permutative reduction), 
      \item it ends with an \(\RuleNameIndE\) rule instance whose main\fixme{define} term is either \(\num 0\) or \(\succ(\ltb)\) for some \(\ltb\) (\(\RedNameInd\) reduction), 
      \item it ends with an \(\EM\) rule instance \(\riA\), 
        in the subderivation of the leftmost premiss of \(\riA\)
        one occurrence of the assumption discharged by \(\riA\) is the premiss of a \(\RuleName\forall{E}\) rule instance \(\riB\) and
        the conclusion of \(\riB\) is an occurrence of a closed formula (\(\RedNameWitness\)), 
      \item it ends with an elimination or atomic rule instance \(\riA\) and one premiss of \(\riA\) is the conclusion of a \(\EM\) rule instance (\(\EM\) permutative reduction), 
      \item it ends with an \(\EM\) rule instance \(\riA\) and one of the premisses of \(\riA\) is derived without using the assumption discharged by \(\riA\) (\(\EM\) immediate simplification). 
    \end{itemize}
    A derivation is \emph{reducible} when it or one of its subderivations is directly reducible.
    A derivation is \emph{normal} when it is not reducible. 
  \end{definition}
  Since the reductions are meant to remove detours and simplify a derivation, we expect a normal derivation to have a particularly simple structure.

  \subsection{Witness reduction in two steps} \fixme{rewrite}
  Instead of giving the single reduction \(\RedNameWitness\), 
  we can split it in two distinct reductions, 
  one that looks for counterexamples and eliminates occurrences of the open assumptions of the \(\EM\) rule and 
  one that eliminates instances of the \(\EM\) rule when their conclusion can be derived without the universal or existential assumption. 


  More precisely consider an instance \(\riA\) of the \(\EM\) rule for the quantifier-free formula \(\fa\):
  \[
    \PrAss{\quant\forall\lva\fa}\riA
    \PrInf[\derA_1]
    \PrUn\fc
    \PrAss{\lnot\fa\subst\lva\lvb}\riA
    \PrInf[\derA_2]
    \PrUn\fc
    \PrLbl[\riA]\EM
    \PrBin\fc
    \DisplayProof
  \]
  Let \(\foA\) be an occurrence of the assumption \(\quant\forall\lva\fa\) discharged by \(\riA\) in \(\derA_1\)
  In order to be able to perform the reduction we assume the following:
  \begin{itemize}
    \item \(\foA\) is the premiss of a \(\RuleName\forall{E}\) instance \(\riB\), 
    \item the conclusion of \(\riB\) is the occurrence of a closed formula. 
  \end{itemize}
  Let \(\foB\) be the conclusion of \(\riB\). 
  \(\foB\) is an occurrence of the closed quantifier-free formula \(\fa\subst\lva\lta\) for some term \(\lta\). 
  Since closed quantifier-free formula are decidable, 
  in the reduction we can distinguish two cases,  
  depending on whether \(\fa\subst\lva\lta\) is true or false. 
  \begin{itemize}
    \item \(\fa\subst\lva\lta\) is true. 

      Let \(\derA_1\) be a derivation of \(\fa\subst\lva\lta\) and 
      replace \(\riB\) with \(\derA_1\):
      \[ 
        \PrAss{\quant\forall\lva\fa}\riA
        \PrLbl{\RuleName\forall{E}}
        \PrUn{\fa\subst\lva\lta}
        \PrInf
        \DisplayProof \quad \leadsto \quad 
        \PrAx{}
        \PrInf[\derA_1]
        \PrUn{\fa\subst\lva\lta}
        \PrInf
        \DisplayProof
      \] 

    \item otherwise \(\lnot\fa\subst\lva\lta\) is true.

      Let \(\derA_2\) be a derivation of \(\lnot\fa\subst\lva\lta\) and 
      replace all the occurrences of the assumption \(\lnot\fa\subst\lva\lta\) discharged by \(\riA\) in the derivation of its rightmost premiss with \(\derA_2\):
      \[ 
        \PrAss{\lnot\fa\subst\lva\lvb}\riA
        \PrInf
        \DisplayProof \quad \leadsto \quad
        \PrAx{}
        \PrInf[\derA_2]
        \PrUn{\lnot\fa\subst\lva\lta}
        \PrInf
        \DisplayProof
      \] 
  \end{itemize}
  We denote this reduction as \(\RedNameWitness\). 

  Whenever this reduction can be applied it removes one or more occurrences of one of the assumptions discharged by \(\riA\). 
  If there are no more occurrences of such assumptions in either \(\derA_1\) or \(\derA_2\) then \(\riA\) is redundant and can be deleted by \(\RedNameSimpl\EM\). 

  \subsection{Immediate simplification}
  Redundant \(\EM\) rule instances can be defined and reduced in the same way as redundant \(\RuleName\lor{E}\) instances. 
  Consider an \(\EM\) rule instance \(\riA\) such that one of its premisses is derived without using the assumption discharged by \(\riA\).
  Then we can reduce as follows:
  \[
    \PrAss{\quant\forall\lva \fa}\riA
    \PrInf[\derA_1]
    \PrUn\fc
    \PrAss{\lnot \fa\subst\lva\lvb}\riA
    \PrInf[\derA_2]
    \PrUn\fc
    \PrLbl[\riA]\EM
    \PrBin\fc
    \DisplayProof
    \leadsto 
    \PrAx{} 
    \PrInf[\Sigma_i] 
    \PrUn\fc 
    \DisplayProof 
  \]
  depending on whether it is \(\derA_1\) or \(\derA_2\) that contains no occurrence of the assumption discharged by \(\riA\). 
  We denote this reduction as \(\RedNameSimpl\EM\). 

  \newcommand\pathA{\mathfrak{a}}
  \newcommand\pathB{\mathfrak{b}}
  \begin{definition}[Path]
    A \emph{path} in a derivation \(\Pi\) is a sequence \(\foA_1, \dotsc, \foA_n\) of formula occurrences in \(\Pi\) such that:
    \begin{itemize}
      \item \(\foA_1\) is the conclusion of an atomic axiom instance or an assumption that is not discharged by a \(\RuleName\lor{E}\) or \(\RuleName\exists{E}\) rule instance; 
      \item for all \(i < n\), \(\foA_i\) and \(\foA_{i+1}\) are respectively a premise and the conclusion of the same rule instance, except when \(\foA_i\) is the major premise of an instance \(\riA\) of a \(\RuleName\lor{E}\) or \(\RuleName\exists{E}\) rule: 
        in this case \(\foA_{i+1}\) is an occurrence of an assumption discharged by the rule instance;
      \item \(\foA_n\) is either the conclusion of \(\Pi\) or the minor premise of a \(\RuleName\limply{E}\) and no \(\foA_i\) with \(i < n\) is such a minor premiss.
    \end{itemize}
  \end{definition} 

  \begin{definition}
    A path \(\foA_1, \dotsc, \foA_n\) of occurrences of formulas \(\fa_1, \dotsc,\fa_n\)
    \begin{itemize}
      \item 
        there are \( 1 \leq m_1 \leq m_2 \leq n\) such that 
        \begin{itemize}
          \item 
            \(\foA_i\) is the conclusion of an elimination rule instance
            for all \( 1 < i \leq m_1\),
          \item 
            \(\foA_i\) is the conclusion of an atomic rule instance 
            for all \( m_1 < i \leq m_2\),
          \item 
            \(\foA_i\) is the conclusion of an introduction rule instance 
            for all \( m_2 < i \leq n\),
        \end{itemize}
      \item 
        has the subformula property when \(\fa_i\) is a subformula

    \end{itemize}
  \end{definition}

  \begin{theorem}
    Let \(\derA\) be a derivation in HA. 
    Then all paths in \(\derA\) have are normal or 
    \(\derA\) is reducible.
  \end{theorem}
  \begin{proof}
    By induction on the structure of \(\Pi\). 

    Consider any path \(\pathA = \foA_1, \dotsc, \foA_n\) of \(\Pi\).
    If \(\foA_n\) is not the conclusion of \(\Pi\) then \(\pathA\) is contained in a subderivation \(\derB\) of \(\derA\). 
    By inductive hypothesis either \(\pathA\) is normal or \(\derB\) is reducible and then so is \(\derA\) and we are done.

    Otherwise \(\foA_n\) is the conclusion of \(\Pi\). 
    Let \(\riA\) be the rule instance \(\foA_n\) is the conclusion of. 
    If \(\riA\) is an atomic axiom then \(n = 1\) and \(\pathA\) is normal.
    Otherwise we can assume that \(n > 1\). 
    Let \(\riB\) be the rule instance \(\foA_{n-1}\) is the conclusion of and let \(\derB\) the derivation of \(\foA_{n-1}\). 
    Let \(\pathB\) be the path \(\foA_1, \dotsc, \foA_{n-1}\) of \(\derB\). 
    \[
      \PrAx\dotso
      \PrAx\dotso
      \PrLbl[\riB]{}
      \PrUn{\fa_{n-1}}
      \PrAx\dotso
      \PrLbl[\riA]{}
      \PrTri{\fa_n}
      \DisplayProof
    \]
    We consider various cases depending on which rule \(\riA\) is an instance of.
    \begin{description}
      \item[atomic] 
        If \(\riA\) is an atomic rule then 
      \item[introduction]
        If \(\riA\) is an introduction rule instance 
      \item[\(\riA\) is atomic]
      \item[\(\riA\) is atomic]
    \end{description}
  \end{proof}

\end{thesis}

\end{document}